\documentclass[a4paper]{amsart} 
\usepackage{amssymb,amscd,booktabs,verbatim} 
\usepackage[all]{xy}
\SelectTips{cm}{}
\usepackage{tikz-cd}
\usepackage{spverbatim}
\usepackage{pdfsync}
\usepackage{hyperref}
\calclayout

\title{A categorification of non-crossing partitions}

\dedicatory{Dedicated to the memory of Dieter Happel.}
\thanks{Version from May 19, 2015.}

\author{Andrew Hubery}
\address{Andrew Hubery\\Fakult\"at f\"ur Mathematik\\
  Universit\"at Bielefeld\\ D-33501 Bielefeld\\ Germany.}
\email{hubery@math.uni-bielefeld.de}
\author{Henning Krause}
\address{Henning Krause\\ Fakult\"at f\"ur Mathematik\\
  Universit\"at Bielefeld\\ D-33501 Bielefeld\\ Germany.}
\email{hkrause@math.uni-bielefeld.de}

\newtheorem{lem}{Lemma}[section]
\newtheorem{prop}[lem]{Proposition}
\newtheorem{cor}[lem]{Corollary}
\newtheorem{thm}[lem]{Theorem}

\theoremstyle{remark}
\newtheorem{rem}[lem]{Remark}

\theoremstyle{definition}
\newtheorem{exm}[lem]{Example}

\numberwithin{equation}{section}

\renewcommand{\mod}{\operatorname{mod}\nolimits}

\newcommand{\Aut}{\operatorname{Aut}\nolimits}

\newcommand{\proj}{\operatorname{proj}\nolimits}
\newcommand{\rad}{\operatorname{rad}\nolimits}

\newcommand{\add}{\operatorname{add}\nolimits}

\newcommand{\id}{\operatorname{id}\nolimits}

\newcommand{\End}{\operatorname{End}\nolimits}
\newcommand{\Hom}{\operatorname{Hom}\nolimits}

\renewcommand{\Im}{\operatorname{Im}\nolimits}
\newcommand{\Ker}{\operatorname{Ker}\nolimits}

\newcommand{\diag}{\operatorname{diag}\nolimits}
\renewcommand{\dim}{\operatorname{dim}\nolimits}

\newcommand{\Exc}{\operatorname{Exc}\nolimits}
\newcommand{\Ext}{\operatorname{Ext}\nolimits}

\newcommand{\Sub}{\operatorname{Sub}\nolimits}

\newcommand{\height}{\operatorname{ht}\nolimits}

\newcommand{\dimv}{\operatorname{\mathbf{dim}}}

\newcommand{\rk}{\operatorname{rk}\nolimits}

\newcommand{\cox}{\operatorname{cox}\nolimits}
\newcommand{\NC}{\operatorname{NC}\nolimits}

\newcommand{\op}{\mathrm{op}}

\newcommand{\can}{\mathrm{can}}

\newcommand{\lto}{\longrightarrow}
\newcommand{\xto}{\xrightarrow}

\def\a{\alpha}

\def\e{\varepsilon}
\def\d{\delta}

\def\p{\phi}

\def\s{\sigma}
\def\t{\tau}

\def\De{\Delta}
\def\Ga{\Gamma}

\def\C{{\mathsf C}}
\def\D{{\mathsf D}}

\def\bbZ{{\mathbb Z}}

\newcommand{\dimvector}[7]{\displaystyle\genfrac{}{}{0pt}{}{#1}{#2}#3\,#4\,#5\genfrac{}{}{0pt}{}{#6}{#7}}

\begin{document}

\begin{abstract}
We present a categorification of the non-crossing partitions given by crystallographic Coxeter groups. This involves a category of certain bilinear lattices, which are essentially determined by a symmetrisable generalised Cartan matrix together with a particular choice of a Coxeter element. Examples arise from Grothendieck groups of hereditary artin algebras.
\end{abstract}

\maketitle
\setcounter{tocdepth}{1}
\tableofcontents

\section{Introduction}

It has long been understood that the exceptional objects, or more generally the exceptional sequences and tilting objects, play a central role in understanding categories of modules or sheaves, and more recently also in the theory of cluster categories. Over a finite dimensional hereditary algebra, the dimension vectors of the exceptional modules, the so-called real Schur roots, also occur in the canonical decomposition, and so describe the indecomposable summands of a general module of fixed dimension vector. In this setting it is therefore of interest to be able to determine combinatorially the subset of real Schur roots inside the set of all real roots. Work in this direction includes \cite{Sc1990,Sc1992,DW2002}.

Inspired by \cite{Le1996} we introduce the notion of generalised Cartan lattice $(\Gamma,E)$, which is a lattice $\Gamma$ equipped with a non-degenerate bilinear form, together with a choice of orthogonal exceptional sequence $E$. The prototypical example of such a generalised Cartan lattice is the Grothendieck group $K_0(A)$ of a finite dimensional hereditary algebra $A$ equipped with the Euler form, together with the classes of the simple modules (suitably ordered). Each such lattice has an associated symmetrisable generalised Cartan matrix, and hence we can define the Weyl group (more precisely the Coxeter system) $W(\Gamma,E)$ and the set of real roots $\Phi(\Gamma,E)$. We also have a natural choice of Coxeter element, denoted $\cox(\Gamma)$, and thus the poset of non-crossing partitions $\NC(\Gamma,E)$.

We recall that non-crossing partitions were introduced by Kreweras \cite{Kr1972} and later generalised in the context of Coxeter groups by Brady and Watt \cite{Br2001,BW2002}, and independently by Bessis \cite{Be2003}; see also Armstrong's memoir \cite{Ar2009}. One connection between non-crossing partitions and representations of quivers has already been pointed out by Ingalls and Thomas \cite{IT2009}; it arises from the categorification of cluster algebras \cite{FZ02} via cluster categories \cite{BMRT07,MRZ03}.

We also introduce the notion of a (mono-)morphism between generalised Cartan lattices, and thus construct the category $\mathfrak C$. We then show that the map $(\Ga,E)\mapsto W(\Ga,E)$, sending a generalised Cartan lattice to its associated Weyl group, is functorial. More precisely, we have the following result, summarising Theorems~\ref{th:main} and \ref{th:natural}.

\begin{thm}
Let $\p\colon(\Gamma',E')\to(\Gamma,E)$ be a morphism of generalised Cartan lattices.
\begin{enumerate}
\item The map $\p$ restricts to an inclusion $\Phi(\Gamma',E')\to\Phi(\Gamma,E)$.
\item There is an injective group homomorphism $\p_*\colon W(\Gamma',E')\to W(\Gamma,E)$, acting on reflections as $s_a\mapsto s_{\p(a)}$.
\item The map $\p_*$ identifies $\NC(\Gamma',E')$ with the subposet $\{w\le\p_*(\cox(\Gamma'))\}$ of $\NC(\Gamma,E)$.
\end{enumerate}
\end{thm}

This theorem is an analogue (for Weyl groups of symmetrisable Kac--Moody Lie algebras) of a result of Bessis \cite{Be2003} which describes for finite Coxeter groups the non-crossing partitions as Coxeter elements of parabolic subgroups. However, $W(\Ga',E')$ need not be parabolic when $W(\Ga,E)$ is infinite (Example~\ref{ex:parabolic}). It turns out that the subgroups of $W(\Ga,E)$ arising from subobjects of $(\Ga',E')\subseteq(\Ga,E)$ form a distinguished class of subgroups which are determined by their Coxeter elements $\cox(\Ga')$ (Corollary~\ref{co:subgroup}).

We can also regard our results as providing a combinatorial model for the category $\mathfrak H$ of hereditary abelian categories arising in the represention theory of algebras. More precisely, the objects in $\mathfrak H$ are the categories $\mod A$ of finitely generated modules over an hereditary artin algebra $A$. The morphisms in $\mathfrak H$ are fully faithful exact functors, modulo natural isomorphisms, having an extension closed essential image.

The map sending an abelian category to its Grothendieck group yields a faithful functor
\[ \mathfrak H \lto \mathfrak C, \qquad \mod A \mapsto K_0(A), \]
and provides the link between representation theory and combinatorics (Corollary~\ref{co:faithful}). Applying our results to $\mod A$ we now obtain the following formulation (Corollary~\ref{co:epi}) of a result by Ingalls--Thomas \cite{IT2009} and Igusa--Schiffler \cite{IS2010}.

\begin{thm}
Let $A$ be a finite dimensional hereditary algebra. Let $\Sub(\mod A)$ denote the poset of subcategories of $\mod A$ of the form $\C(X)$ for some exceptional sequence $X$, ordered by inclusion, and let $\NC(K_0(A))$ be the poset of non-crossing partitions attached to the generalised Cartan lattice $K_0(A)$. Then there is a natural isomorphism of posets $\Sub(\mod A)\cong\NC(K_0(A))$ sending the subcategory $\C(X)$ to the non-crossing partition $\cox(\C(X))$.

In particular, two exceptional sequences $X$ and $Y$ are equivalent under the braid group action if and only if they determine the same non-crossing partition.
\end{thm}

Note that this point-of-view is also apparent in the work of Happel \cite{Ha1987}, see Theorem~\ref{th:Happel}. Also, the study of the categories $\C(X)\subset\mod A$ is quite natural, since they can be characterised in a number of different ways: they are the thick subcategories such that the inclusion admits a left or right adjoint; or as the thick subcategories either generated by, or perpendicular to, an exceptional sequence; or the subcategories obtained by restriction of scalars along a homologial epimorphism, see Theorem \ref{th:subcat}. In particular, all finitely generated thick subcategories arise in this way, see Remark \ref{re:subcat}.

Much of the proofs of these theorems can be done entirely in the language of generalised Cartan lattices, exploiting the transitive braid group action on factorisations of the Coxeter element \cite{BDSW14,IS2010}. In particular, we introduce the notion of a \emph{real exceptional sequence}, and use these to define the morphisms in $\mathfrak C$. We then show in Proposition~\ref{pr:real-seq} that the map $F\mapsto\cox(F)$ determines a surjective map from real exceptional sequences to non-crossing partitions. However, we do not know of any combinatorial proof of the facts that the fibres of this map are precisely the orbits under the braid group action, and that each fibre contains an orthogonal exceptional sequence.

To prove these two results we need that every generalised Cartan lattice arises as the Grothendieck group $K_0(A)$ of an hereditary artin algebra $A$. We then show that under any such realisation, the real exceptional sequences in $K_0(A)$ correspond precisely to the exceptional sequences in $\mod A$, Proposition~\ref{pr:exceptional-ncp}. We can then apply the theory of perpendicular categories to finish the proof.

As an application to Coxeter systems, we note that in \cite{BDSW14} it is shown that the factorisations of a \emph{parabolic} Coxeter element form a single orbit under the braid group action. It follows however from our results that the factorisations of \emph{any} non-crossing partition form a single orbit, and moreover there is one factorisation which forms a simple system.

As an application to representation theory, we show that the set of dimension vectors of exceptional $A$-modules depends only on $K_0(A)$, leading in turn to an essentially root-theoretic proof of Gabriel's Theorem~\ref{th:Gabriel}. This answers the question posed by Gabriel in \cite[Section~4]{Ga1973}, but now for all Dynkin types, not just $ADE$-type. We also show that the theorem of Crawley-Boevey \cite{CB1992} and Ringel \cite{Ri1994} is a consequence of the transitivity of the braid group action for Coxeter systems.

We also give an algorithm, based on the work of Schofield (Proposition~\ref{pr:Schofield}) and Derksen--Weyman \cite{DW2002}, of how one can check whether a given exceptional sequence of (pseudo-real) roots is actually a real exceptional sequence. An explicit example of this is given in Example~\ref{exm:Schofield}.

In the last section we also relate our approach to the study of Hom-free sets, which are collections of exceptional objects having pairwise  only zero homomorphisms. In finite representation type there is an obvious bijection between the two points of view, given by sending a subcategory $\C$ closed under kernels, cokernels and extensions to its set of simple objects (Proposition~\ref{pr:hom-free}). It is therefore interesting to note that this approach linking Catalan combinatorics and the representation theory of algebras was already observed in the early 1980s by Gabriel and his school \cite{Rie1980,Ga1981}. In \cite{GP1987} Gabriel and de la Pe\~na counted the Hom-free sets of indecomposable modules for Dynkin quivers and obtained the Coxeter-Catalan numbers of $ADE$-type. On the other hand, Riedtmann used such sets to classify the representation-finite self-injective algebras of type $A$ \cite{Rie1980}.

For another intriguing correspondence between representations of hereditary algebras and Weyl group elements see \cite{ORT2012}.

For the convenience of the reader we include in the appendix a survey of the perpendicular calculus, as well as a discussion on the various notions of crystallographic Coxeter groups.

\subsection*{Acknowledgements}
The idea for this work goes back to an Oberwolfach meeting in 2005 when the results of Ingalls and Thomas involving generalised non-crossing partitions were presented. Since then the second named author discussed this with many colleagues -- too many to be listed here -- and received valuable comments. He wishes to express special thanks to Christof Gei{\ss}, Lutz Hille, Claus Michael Ringel, Christian Stump, and Hugh Thomas. Both authors would also like to thank the referee for their careful reading of the article and helpful comments.

\section{Bilinear lattices and exceptional sequences}

The Grothendieck group of an abelian or triangulated category is an abelian group with some additional structure given by the corresponding bilinear Euler form. In this section we provide an axiomatic treatment which is inspired by work of Lenzing on Grothendieck groups of canonical algebras \cite{Le1996}. In particular, the following definition of a bilinear lattice is taken from there.  We then consider exceptional sequences and the action of the braid group in this setting, modelling their properties in the Grothendieck group of an abelian or triangulated category. Exceptional sequences were introduced in the Moscow school of vector bundles, see for instance \cite{Bo1990,Go1988,GR1987,Ru1990}; later they appeared in representation theory \cite{CB1992,Ri1994}.  The axiomatic treatment in the context of bilinear lattices seems to be new.

\subsection*{Bilinear lattices}

A \emph{bilinear lattice} is an abelian group $\Ga$ together with a non-degenerate bilinear form
\[ \langle -,-\rangle\colon \Ga\times \Ga\lto \bbZ. \]
Here, \emph{non-degenerate} means that $\langle x,-\rangle=0$ implies $x=0$, and $\langle -,y\rangle=0$ implies $y=0$.  Note that $\Ga$ is torsion-free.  The corresponding \emph{symmetrised form} is
\[ (x,y) = \langle x,y\rangle+\langle y,x\rangle\quad\text{for } x,y\in \Ga. \]
For a subset $S$ of $\Ga$ one defines the right and left orthogonal complements
\begin{align*} 
S^\perp &:= \{x\in \Ga\mid \langle s,x\rangle=0\text{ for all }s\in S\}\\
^\perp S &:= \{x\in \Ga\mid \langle x,s\rangle=0\text{ for all }s\in S\}
\end{align*}
In the following $\Ga$ denotes a bilinear lattice.  

\subsection*{Roots}

An element $a\in \Ga$ is called a \emph{pseudo-real root}, or just a \emph{root}, if $\langle a,a\rangle >0$ and $\frac{\langle   a,x\rangle}{\langle a,a\rangle}, \frac{\langle x,a\rangle}{\langle a,a\rangle}\in\bbZ$ for all $x\in \Ga$. For such a root $a$ one has the following transformations:
\begin{align*}
l_a\colon &\Ga\lto \Ga,\qquad x\mapsto x-\frac{\langle a,x\rangle}{\langle a,a\rangle}a\\
r_a\colon &\Ga\lto \Ga,\qquad x\mapsto x-\frac{\langle x,a\rangle}{\langle a,a\rangle}a\\
s_a\colon &\Ga\lto \Ga,\qquad x\mapsto x-2\frac{(x,a)}{(a,a)}a
\end{align*}
Note that $r_a$ and $l_a$ are adjoint with respect to the bilinear from, in the sense that
\[ \langle r_a(x),y \rangle = \langle x,l_a(y)\rangle \quad\textrm{for all }x,y\in\Ga, \]
and each $s_a$ is a reflection, so fixes a subgroup of corank one and sends $a\mapsto-a$.

If $\Ga'$ is another bilinear lattice, with bilinear form $\langle-,-\rangle'$, then an \emph{isometry} $\p\colon\Ga'\to\Ga$ is a group homomorphism preserving the bilinear forms, so $\langle\p(x),\p(y)\rangle=\langle x,y\rangle'$ for all $x,y\in\Ga'$.

We will also need the group $\Aut(\Ga):=\Aut(\Ga,(-,-))$, the group of all automorphisms of $\Ga$ preserving the \emph{symmetrised} bilinear form.

\begin{lem}\label{le:roots}
Let $a\in\Ga'$ be a root and let $\p\colon\Ga'\to\Ga$ be a group homomorphism preserving the symmetrised bilinear forms. Then $s_{\p(a)}\p=\p s_a$ (as maps $\Ga'\to\Ga$).

In particular, if $a,b\in\Ga$ are roots, then so too is $s_b(a)$ and $s_{s_b(a)} = s_bs_as_b$.
\end{lem}

\begin{proof}
Straightforward computations, where for the second statement we put $\Ga'=\Ga$ and $\p=s_b$. Note that in the first part we have abused notation somewhat, since $\p(a)\in\Ga$ need not be a root, but $s_{\p(a)}$ is well-defined on the image of $\p$.
\end{proof}

The \emph{radical} of $\Ga$ is by definition
\[ \rad \Ga := \{x\in \Ga\mid \langle x,-\rangle=-\langle -,x\rangle\} = \{ x \mid (x,-)=0 \}. \]
This is clearly invariant under $\Aut(\Ga)$.

\subsection*{Exceptional sequences}

A sequence of roots $E=(e_1,\ldots, e_r)$ in $\Ga$ is called \emph{exceptional} of \emph{length} $r$ if $\langle e_i,e_j\rangle=0$ for all $i>j$. The sequence $E$ is \emph{complete} if $\bbZ E=\Ga$. The \emph{empty sequence} $E=\varnothing$ is exceptional of length zero. An exceptional sequence of length two is also called an \emph{exceptional pair}.

Given a sequence of roots $E=(e_1,\ldots, e_r)$ we write
\[ l_E := l_{e_1}\cdots l_{e_r} \qquad r_E := r_{e_r}\cdots r_{e_1} \qquad s_E := s_{e_1}\cdots s_{e_r} \]
and denote by $\bbZ E$ the subgroup of $\Ga$ generated by $e_1,\ldots,e_r$. We observe that, for each $x\in\Ga$, the following all lie in $\bbZ E$
\[ l_E(x)-x, \quad r_E(x)-x, \quad s_E(x)-x. \]
Of particular interest are the transformations $s_E$ for exceptional sequences $E$.

We begin with some elementary observations.

\begin{lem}\label{le:lin-indept}
Let $E=(e_1,\ldots,e_r)$ be an exceptional sequence in $\Gamma$. Then $\bbZ E\cap E^\perp=0$. In particular, the $e_i$ are linearly independent, so $\bbZ E$ has rank $r$.
\end{lem}

\begin{proof}
Take $\sum_ia_ie_i\in\bbZ E\cap E^\perp$ and apply $\langle e_i,-\rangle$ for $i=r,\ldots,1$ in turn.
\end{proof}

\begin{lem}\label{le:except}
Let $E=(e_1,\ldots, e_r)$ be an exceptional sequence in $\Ga$. Then the following hold:
\begin{enumerate}
\item $l_E(x)\in E^\perp$ for $x\in \Ga$ and $l_E(x)=x$ for $x\in E^\perp$.
\item $l_E(x)=0$ iff $x\in \bbZ E$.
\item $\Ga=\bbZ E\oplus E^\perp$.
\item $\langle l_E(x),y\rangle= \langle x,y\rangle$ for $x\in \Ga$ and $y\in E^\perp$.
\item $s_E(x)=l_E(x)$ for $x\in {^\perp E}$.
\end{enumerate}
In particular, $l_E$ is the projection from $\Ga$ onto $E^\perp$ along $\bbZ E$.
\end{lem}

\begin{proof}
The proofs are by induction on $r$. Set $E'=(e_1,\ldots,e_{r-1})$.

(1) Let $x\in \Ga$. We have $r_{E'}(e_r)=e_r$, so
\[ \langle e_r,l_{E'}l_{e_r}(x)\rangle = \langle r_{E'}(e_r),l_{e_r}(x)\rangle = \langle e_r,l_{e_r}(x)\rangle = 0. \]
Thus $l_E(x)\in(E')^\perp\cap e_r^\perp=E^\perp$.

If $x\in E^\perp$, then $l_{e_i}(x)=x$ for all $i$ by definition, so $l_E(x)=x$. 

(2) If $x\in\bbZ E$, then
\[ l_E(x) \in \bbZ E \cap E^\perp = 0.\]
If $l_E(x)=0$, then $l_{e_r}(x)\in\bbZ E'$ by induction, and therefore
\[ x = l_{e_r}(x)+\frac{\langle e_r,x\rangle}{\langle e_r,e_r\rangle}e_r\in\bbZ E. \]

(3) This follows from (1) and (2).

(4) Use that $l_E(x)-x\in\bbZ E$ for all $x\in \Ga$.

(5) If $x\in {^\perp E}$, then $s_{e_i}(x)=l_{e_i}(x)$ for all $i$ by definition.
\end{proof}

\begin{prop}\label{pr:cox}
Let $E=(e_1,\ldots, e_r)$ be an exceptional sequence in $\Ga$ and $x\in \Ga$. Then
\[ \langle y,s_E(x)\rangle = 
\begin{cases}
-\langle x,y\rangle\quad &\text{if } y\in\bbZ E,\\
\langle y,x\rangle\quad &\text{if } y\in {^\perp E}.
\end{cases} \]
\end{prop}

\begin{proof}
For $y\in {^\perp E}$ use that $s_E(x)-x\in\bbZ E$. Suppose therefore that $y\in\bbZ E$. The proof is by induction on $r$. Set $E'=(e_1,\ldots,e_{r-1})$. If $y\in\bbZ E'$, then $e_r\in{^\perp y}$, so
\[ \langle y,s_E(x)\rangle = \langle y,s_{E'}s_{e_r}(x)\rangle = -\langle s_{e_r}(x),y\rangle = -\langle x,y\rangle. \]
If $y\in\bbZ e_r$, then since $s_{E'}(z)-z\in\bbZ E'\subseteq e_r^\perp$ for all $z$, we have
\[ \langle y,s_{E'}s_{e_r}(x)\rangle = \langle y,s_{e_r}(x)\rangle, \]
and a direct computation shows that this equals $-\langle x,y\rangle$.
\end{proof}

\subsection*{The Coxeter transformation}

Let $\Ga$ be a bilinear lattice and suppose that $\Ga$ admits a complete exceptional sequence $E=(e_1,\ldots, e_n)$. The \emph{Coxeter transformation} of $\Ga$ is by definition
\[ \cox(\Ga) := s_E. \]
This does not depend on the choice of $E$ by Proposition~\ref{pr:cox}. Now identify $\Ga=\bbZ^n$ and define an $n\times n$ matrix $C$ by $\langle x,y\rangle=x^t Cy$.

\begin{prop}
The matrix $C$ is invertible. The automorphism of $\Ga$ given by
\[ x \longmapsto c(x) := (-C^{-1}C^t)x \]
equals $\cox(\Ga)$, and satisfies $\langle -,c(x)\rangle = -\langle x,-\rangle$ for all $x\in\Ga$.
\end{prop}

\begin{proof}
We have
\[ \langle x,y\rangle= x^t Cy=y^t C^tx =-y^tC(-C^{-1} C^t)x =-\langle y, c(x)\rangle. \]
Thus $c=\cox(\Ga)$ by Proposition~\ref{pr:cox}.
\end{proof}

\subsection*{The braid group action}

For an integer $r\ge 1$ let $B_r$ be the \emph{braid group} on $r$ strands, so with generators $\s_1,\ldots,\s_{r-1}$ and relations
\begin{alignat*}{2}
\s_i\s_j &= \s_j\s_i &&\quad\textrm{for }|i-j|>1\\
\s_i\s_{i+1}\s_i &= \s_{i+1}\s_i\s_{i+1} &&\quad\textrm{for }1\le i\le r-2.
\end{alignat*}
We also consider the wreath product $\{\pm1\}\wr B_r$, so the semi-direct product $\{\pm1\}^r\rtimes B_r$ of the braid group with the \emph{sign group}, with multiplication given by
\[ \s_i(\e_1,\ldots,\e_r)\s_i^{-1} := (\e_1,\ldots,\e_{i-1},\e_{i+1},\e_i,\e_{i+2},\ldots,\e_r). \]

\begin{prop}\label{pr:braidaction} 
Let $r\ge 1$ be an integer. Then the wreath product $\{\pm1\}\wr B_r$ acts on exceptional sequences of length $r$ via
\begin{align*}
\s_i(e_1,\ldots,e_r) &:= (e_1,\ldots,e_{i-1},e_{i+1},s_{e_{i+1}}(e_i),e_{i+2},\ldots,e_r)\\
\s_i^{-1}(e_1,\ldots,e_r) &:= (e_1,\ldots,e_{i-1},s_{e_{i}}(e_{i+1}),e_i,e_{i+2},\ldots,e_r)\\
\e(e_1,\ldots,e_r) &:= (\e_1 e_1,\ldots,\e_r e_r).
\end{align*}
\end{prop}

\begin{proof}
We check that the relations for the braid group hold, the rest being clear. Let $E=(e_1,\ldots,e_r)$ be an exceptional sequence. A quick computation using Lemma~\ref{le:roots} shows that $\s_iE$ and $\s_i^{-1}E$ are again exceptional sequences and that $\s_i\s_i^{-1}E=E=\s_i^{-1}\s_iE$. The identity $\s_i\s_jE=\s_j\s_iE$ for $|i-j|>1$ is immediate. For the identity $\s_i\s_{i+1}\s_iE=\s_{i+1}\s_i\s_{i+1}E$, it is enough to show this when $i=1$ and $r=3$. In this case we have $E=(e,f,g)$ and
\begin{align*}
\s_1\s_2\s_1E &= (g,s_g(f),s_gs_f(e))\\
\s_2\s_1\s_2E &= (g,s_g(f),s_{s_g(f)}s_g(e)).
\end{align*}
Now use the identity $s_{s_g(f)}=s_gs_fs_g$ from Lemma~\ref{le:roots}.
\end{proof}

Note that, if $(e,f)$ is an exceptional pair, then $s_e(f)=l_e(f)$ by Lemma~\ref{le:except}, and dually $s_f(e)=r_f(e)$, so we can express the action of the braid group in terms of the maps $l$ and $r$.

\begin{lem}\label{le:preservation-of-Coxeter}
Let $E$ and $F$ be exceptional sequences in $\Ga$ and $\s\in \{\pm1\}\wr B_r$.
\begin{enumerate}
\item $\bbZ\sigma E=\bbZ E$.
\item If $\bbZ E=\bbZ F$, then $s_E=s_F$. In particular, $s_{\sigma E}=s_E$.
\item If $s_E=s_F$, then $\bbZ E +\rad \Ga=\bbZ F +\rad \Ga$.
\item If $e, f\in \Ga$ are roots and $s_e=s_f$, then $e=\pm f$.
\end{enumerate}
\end{lem}

\begin{proof}
(1) It is clear that $\bbZ\s E=\bbZ E$ for each generator of $B_r$ and each element of the sign group. Thus $\bbZ\s E=\bbZ E$ for all $\s\in\{\pm1\}\wr B_r$.

(2) Suppose that $\bbZ E=\bbZ F$. Then ${^\perp E}={^\perp F}$, so $s_E=s_F$ by Proposition~\ref{pr:cox}, using that the form $\langle-,-\rangle$ is non-degenerate and $\Ga=\bbZ E\oplus {^\perp E}$.

(3) Suppose that $s_E=s_F$. Given $x\in\bbZ F$ write $x=x'+x''$ with $x'\in\bbZ E$ and $x''\in {^\perp E}$. Using Proposition~\ref{pr:cox} we have for all $y\in \Ga$ that
\begin{align*}
-\langle y,x'\rangle -\langle y,x''\rangle&=-\langle y,x\rangle\\
&=\langle x,s_F(y)\rangle\\
&=\langle x,s_E(y)\rangle\\
&=\langle x',s_E(y)\rangle+\langle x'',s_E(y)\rangle\\
&=-\langle y,x'\rangle +\langle x'',y\rangle.
\end{align*}
Thus $x''\in\rad \Ga$, and it follows that $\bbZ F\subseteq \bbZ E+\rad \Ga$. The other inclusion holds by symmetry.

(4) Using (3) we have $e=\a f + r$ with $\a\in\bbZ$ and $r\in\rad\Ga$. Thus
\[ -e = s_e(e) = s_f(e) = e-2\a f \]
and so $r=0$. It follows that $\bbZ e\subseteq \bbZ f$. The other inclusion holds by symmetry.
\end{proof}

\section{Generalised Cartan lattices}

In this section we introduce the main object of interest, namely the category of generalised Cartan lattices, and show how to associate to every generalised Cartan lattice a symmetrisable generalised Cartan matrix, and hence a Weyl group and root system, as well as the poset of non-crossing partitions.

\subsection*{Generalised Cartan lattices}

An exceptional sequence $E=(e_1,\ldots,e_r)$ in a bilinear lattice $\Ga$ is said to be \emph{orthogonal} provided $\langle e_i,e_j\rangle\le 0$ for all $i\neq j$. A \emph{generalised Cartan lattice} is a pair $(\Ga,E)$ consisting of a bilinear lattice $\Ga$ and a complete orthogonal exceptional sequence $E=(e_1,\ldots,e_n)$.

We fix a partial order on $\Ga$ by saying $a\geq0$ provided $a=\sum_i\alpha_ie_i$ with $\alpha_i\geq0$ for all $i$.

If $(\Ga,E)$ is a generalised Cartan lattice, then the matrix
\begin{equation}\label{eq:cartan}
C(\Ga,E) := (\langle  e_i,e_i\rangle^{-1}(e_i,e_j))_{i,j}
\end{equation}
is a symmetrisable generalised Cartan matrix.\footnote{
Following Kac \cite{Ka1990} we call an integral square matrix $C$ a \emph{symmetrisable generalised Cartan matrix} if $c_{ii}=2$ and $C=D^{-1}B$ for some diagonal matrix $D=\diag(d_i)$ and symmetric matrix $B$ with $d_i>0$ and $b_{ij}\le0$ for $i\neq j$.
}

The converse also holds.

\begin{lem}\label{le:gcm}
Every symmetrisable generalised Cartan matrix $C = D^{-1}B$ is of the form $C(\Ga,E)$ for some generalised Cartan lattice $(\Ga,E)$.
\end{lem}

\begin{proof}
Let $C = D^{-1}B$ be a symmetrisable generalised Cartan matrix of size $n$ with $D=\diag(d_i)$. Take $\Ga=\bbZ^n$ with standard basis $\{e_1,\ldots,e_n\}$ and equip this with the bilinear form given by
\[ \langle e_i,e_j\rangle := \begin{cases} b_{ij}&\textrm{if }i<j;\\d_i&\textrm{if }i=j;\\0&\textrm{if }i>j. \end{cases} \]
Then $E=(e_1,\ldots,e_n)$ is a complete, orthogonal exceptional sequence, $(\Ga,E)$ is a generalised Cartan lattice, and $C(\Gamma,E)=C$.
\end{proof}

\subsection*{Weyl groups and non-crossing partitions}

The \emph{Weyl group} $W=W(\Ga,E)$ of a generalised Cartan lattice is defined to be the subgroup of $\Aut(\Ga)$ generated by the \emph{simple reflections} $S:=\{s_{e_1},\ldots,s_{e_n}\}$. Then $(W,S)$ is a Coxeter system \cite[Proposition~3.13]{Ka1990}. In general, a  \emph{Coxeter element} in $(W,S)$ is a product of all the generators in $S$, in some order. Thus $\cox(\Gamma)$ is always a Coxeter element in the Weyl group $W(\Ga,E)$. 

Note that the Weyl group depends only on the Cartan matrix $C(\Ga,E)$, and that different choices of orthogonal exceptional sequences in $\Ga$ can give rise to the same Cartan matrix.

The set of \emph{real roots} is
\[ \Phi = \Phi(\Ga,E) := \{ w(e_i) \mid w\in W(\Ga,E),\,1\le i\le n \} \subseteq\Ga. \]
By Lemma~\ref{le:roots} we see that each real root is a (pseudo-real) root. Moreover, every real root is either positive or negative (combine Theorem~1.2 and Proposition~3.7~(b) from \cite{Ka1990}). Finally, if $a=w(e_i)\in\Phi$, then $-a=ws_{e_i}(e_i)\in\Phi$ and
\[ s_{w(e_i)} = w s_{e_i} w^{-1} \in W(\Ga,E). \]
A \emph{reflection} in $W$ is thus defined to be an element of the form $s_a$ for $a\in\Phi$, so the set of all reflections is
\[ T := \{s_a \mid a\in\Phi\} = \{ wsw^{-1} \mid w\in W,\, s\in S \}. \]

\begin{rem}
The set of reflections depends on the choice of Coxeter system. For example, the dihedral group $D_{12}$ of order 12 has two presentations as a Coxeter group
\begin{align*}
D_{12} &= \langle s,t\mid s^2=t^2=(st)^6=1 \rangle\\
&= \langle s,u,v \mid s,u,v \mid s^2=u^2=v^2=(su)^3=(sv)^2=(uv)^2=1\rangle
\end{align*}
coming from the isomorphism $D_{12}\cong D_6\times C_2$. Note that $t=uv$ and $v=(st)^3$, so that the Coxeter elements agree, $st=suv$. In the first presentation there are six reflections, whereas there are only four in the second presentation.
\end{rem}

The \emph{absolute length} $\ell(w)$ of $w\in W$ is the minimal $r\ge 0$ such that $w$ can be written as product $w=t_1\ldots t_r$ of reflections $t_i\in T$. The \emph{absolute order} on $W$ is then defined as
\[ u\le v \quad\textrm{provided}\quad \ell (u)+\ell(u^{-1}v) = \ell(v). \]
For another description of this length we refer to \cite{Dy2001}.

Relative to a Coxeter element $c$ one defines the poset of \emph{non-crossing partitions}
\[ \NC(W,c) := \{ w\in W \mid \id \le w\le c \}. \]
When $(\Ga,E)$ is a generalised Cartan lattice, with Weyl group $W$ and Coxeter element $s_E$, we also write $\NC(\Ga,E)$ instead of $\NC(W,s_E)$.

Observe that if $u\leq w$ are non-crossing partitions, say $w=uv$, then since $w=v(v^{-1}uv)$ also $v$ is a non-crossing partition.

The braid group $B_r$ acts on the set of all $r$ element sequences in any group via
\[ \s_i (x_1,\ldots,x_r) := (x_1,\ldots,x_{i-1},x_{i+1}, x_{i+1}^{-1}x_ix_{i+1}, x_{i+2},\ldots,x_r). \]
Note that the product of the group elements remains the same.

The braid group action on factorisations of the Coxeter element is transitive whenever $W$ is a Coxeter group; we record this important
result for later use. For finite Coxeter groups, a proof can be found in a letter of Deligne \cite{De1974}. For the absolute length of the Coxeter element see Dyer \cite{Dy2001}.

\begin{thm}[Igusa--Schiffler \cite{IS2010}, see also \cite{BDSW14}]\label{th:IS2010}
Let $(W,S)$ be a Coxeter system of rank $|S|=n$ and let $c$ be a Coxeter element. Then $\ell(c)=n$ and the braid group $B_n$ acts transitively on the set of sequences of reflections $(t_1,\ldots,t_n)$ such that $t_1\cdots t_n=c$. \qed
\end{thm}

\subsection*{Real exceptional sequences}

Let $(\Ga,E)$ be generalised Cartan lattice of rank $n$.

A \emph{subsequence} of a sequence $(f_1,\ldots,f_r)$ of elements of $\Ga$ is one of the form $(f_{i_1},\ldots,f_{i_s})$ for $1\le i_1<i_2<\cdots<i_s\le r$; it is an \emph{initial subsequence} if $i_j=j$ for all $j$. A \emph{real exceptional sequence} is a subsequence of a complete exceptional sequence $(f_1,\ldots,f_n)$ where each $f_i$ is a real root.

Observe that the action of the wreath product $\{\pm1\}\wr B_r$ on exceptional sequences of length $r$ restricts to an action on real exceptional sequences. This is clear from the definition of the action, using that $\Phi=-\Phi$.

\begin{lem}\label{le:equivar}
The map $(f_1,\ldots,f_r)\mapsto(s_{f_1},\ldots,s_{f_r})$ sending a real exceptional sequence to the sequence of reflections in the Weyl group is equivariant for the action of the braid group $B_r$.
\end{lem}

\begin{proof}
It is enough to check this for the generators $\sigma_i^{\pm1}$, and hence just for $r=2$. The result now follows from the identity $s_{s_e(f)}=s_es_fs_e$ from Lemma~\ref{le:roots}.
\end{proof}

\begin{lem}\label{le:real-seq}
The real exceptional sequences in $(\Ga,E)$ are precisely the initial subsequences of the sequences $\s E$ for $\s\in\{\pm1\}\wr B_n$. In particular, the wreath product $\{\pm1\}\wr B_n$ acts transitively on the set of complete real exceptional sequences.
\end{lem}

\begin{proof}
Using the braid group action it is clear that every real exceptional sequence is an initial sequence of a complete real exceptional sequence. Now let $F=(f_1,\ldots,f_n)$ be any complete real exceptional sequence. Then $s_F=s_E=\cox(\Ga)$, so by Theorem~\ref{th:IS2010} there exists $\s\in B_n$ such that $\s(s_{e_1},\ldots,s_{e_n})=(s_{f_1},\ldots,s_{f_n})$. Consider $(g_1,\ldots,g_n):=\s(e_1,\ldots,e_n)$. Then each $g_i$ is a real root and $s_{f_i}=s_{g_i}$ by the previous lemma, so $f_i=\pm g_i$ by Lemma~\ref{le:preservation-of-Coxeter}~(4).
\end{proof}

The following relates real exceptional sequences to non-crossing partitions, and improves upon Lemma~\ref{le:preservation-of-Coxeter}~(3).

\begin{prop}\label{pr:real-seq}
Let $(\Ga,E)$ be a generalised Cartan lattice and consider the map $F\mapsto s_F$ from real exceptional sequences to the Weyl group. Then the image is precisely $\NC(\Ga,E)$, and $s_F=s_{F'}$ if and only if $\bbZ F=\bbZ F'$.
\end{prop}

\begin{proof}
Let $F$ be a real exceptional sequence. By the previous lemma $F$ is an initial subsequence of $\sigma E$ for some $\sigma\in\{\pm1\}\wr B_n$, say $\sigma E=(F,G)$. Then $\cox(\Gamma)=s_E=s_Fs_G$, and so $s_F\in\NC(\Gamma,E)$.

Next let $w\in\NC(\Gamma,E)$ be a non-crossing partition. By Theorem~\ref{th:IS2010} we can write $w=t_1\cdots t_r$ with $(t_1,\ldots,t_n)=\sigma(s_{e_1},\ldots,s_{e_n})$ for some $\sigma\in B_n$. Set $(f_1,\ldots,f_n):=\sigma(e_1,\ldots,e_n)$ and $F=(f_1,\ldots,f_r)$. Then $F$ is a real exceptional sequence, $t_i=s_{f_i}$ for all $i$, and $w=s_F$.

Finally, suppose $s_F=s_{F'}$ for two real exceptional sequences $F$ and $F'$. Set $w=s_F^{-1}\cox(\Gamma)$. Then $w$ is again a non-crossing partition (since $\cox(\Gamma)=wv$ for $v=w^{-1}s_Fw$), so $w=s_G$ for some real exceptional sequence $G$. It follows that both $(F,G)$ and $(F',G)$ are complete real exceptional sequences, and hence that $\bbZ F=G^\perp=\bbZ F'$.
\end{proof}

\begin{rem}\label{re:need-hered}
One of the main results in this article, Theorem~\ref{th:main}, is that the fibres of the map $F\mapsto s_F$ are precisely the orbits of the braid group, and moreover each such orbit contains an orthogonal real exceptional sequence (so the sublattice $\bbZ F$ is naturally a generalised Cartan lattice). To prove this, however, we will need to relate generalised Cartan lattices to Grothendieck groups of hereditary algebras.
\end{rem}

\subsection*{Morphisms of generalised Cartan lattices}

A \emph{morphism}\footnote{
We do not know a reasonable definition of a morphism which covers, for instance, morphisms of the form $K_0(A')\to K_0(A)$ induced by an arbitrary algebra homomorphism $A\to A'$; cf.\ Theorem~\ref{th:epi}.
}
$\p\colon(\Ga',E')\to(\Ga,E)$ between generalised Cartan lattices is an isometry $\p\colon\Ga'\to\Ga$ such that $\p E'$ is a real exceptional sequence in $\Gamma$.

We observe that every such morphism $\p$ necessarily preserves the bilinear form $\langle-,-\rangle$, not just the symmetric form $(-,-)$.

\begin{lem}
The generalised Cartan lattices form a category in which all morphisms are monomorphisms. \qed
\end{lem}

The following gives a more conceptual description of the morphisms.

\begin{prop}\label{pr:morph}
The morphisms $(\Ga',E')\to(\Ga,E)$ between generalised Cartan lattices are precisely those isometries $\p\colon\Ga'\to\Ga$ sending real exceptional sequences in $\Ga'$ to real exceptional sequences in $\Ga$. \qed
\end{prop}

We observe that a morphism $\p\colon(\Ga',E')\to(\Ga,E)$ maps $\Phi(\Gamma',E')$ into $\Phi(\Gamma,E)$. It is not \emph{a priori} clear, however, that we have induced morphisms $W(\Ga',E')\to W(\Ga,E)$ and $\NC(\Ga',E')\to\NC(\Ga,E)$. This functoriality will be established in \S\ref{se:subobjects}.

\subsection*{Real roots for generalised Cartan lattices}

We discuss the difference between pseudo-real and real roots.

\begin{lem}
Let $E=(e_1,\ldots,e_n)$ be a complete exceptional sequence in a bilinear lattice $\Ga$. Then $a=\sum_i\a_ie_i$ is a pseudo-real root if and only if $\langle a,a\rangle>0$ and $\a_i\frac{\langle e_i,e_i\rangle}{\langle a,a\rangle}\in\bbZ$ for all $i$.

In particular, if $\langle e_i,e_i\rangle=1$ for all $i$, so $C(\Gamma,E)$ is symmetric, then $a$ is a pseudo-real root if and only if $\langle a,a\rangle=1$.
\end{lem}

\begin{proof}
Let $a=\sum_i\a_ie_i$ be a pseudo-real root. We want to show that $\a_i\langle e_i,e_i\rangle\in\langle a,a\rangle\bbZ$ for all $1\leq
  i\leq n$. Clearly
\[ \a_1\langle e_1,e_1\rangle = \langle a,e_1\rangle \in \langle a,a\rangle\bbZ, \]
so the result holds for $i=1$. Now let $i>1$ and assume by induction that the result holds for all $j<i$. Since $e_j$ is a root we know that $\langle e_j,e_i\rangle\in\langle e_j,e_j\rangle\bbZ$, so, by induction, for each $j<i$ we can find an integer $\lambda_j$ such that $\a_j\langle e_j,e_i\rangle=\lambda_j\langle a,a\rangle$. Since $\langle a,e_i\rangle\in\langle a,a\rangle\bbZ$, the same holds for
\[ \a_i\langle e_i,e_i\rangle = \langle a,e_i\rangle-\sum_{j<i}\a_j\langle e_j,e_i\rangle = \langle a,e_i\rangle-\sum_{j<i}\lambda_j\langle a,a\rangle. \]

Conversely, suppose $a=\sum_i\a_ie_i$ satisfies $\langle a,a\rangle>0$ and $\a_i\langle e_i,e_i\rangle\in\langle a,a\rangle\bbZ$ for all $i$. Fix $r\in\{1,\ldots,n\}$. Again, since $e_i$ is a root, we have $\a_i\langle e_i,e_r\rangle,\a_i\langle e_r,e_i\rangle\in\langle
  a,a\rangle\bbZ$. It follows that $\frac{\langle a,e_r\rangle}{\langle a,a\rangle},\frac{\langle e_r,a\rangle}{\langle a,a\rangle}\in\bbZ$, and hence $a$ is a pseudo-real root.

For the second statement, it is clear that $\langle a,a\rangle=1$ implies that $a$ is a pseudo-real root. Suppose therefore that $a=\sum_i\alpha_ie_i$ is a pseudo-real root and that $\langle e_i,e_i\rangle=1$ for all $i$. Then by considering $\langle a,e_i\rangle$ for $i=1,\ldots,n$ in turn, we deduce that $d:=\langle a,a\rangle>0$ divides each $\alpha_i$, so $d^2$ divides $d$, and hence $d=1$.
\end{proof}

Using this, it is easy to see that in general there are roots which are not real.

\begin{exm} 
Consider the generalised Cartan lattice $\Ga=\bbZ^4$ with bilinear form given by the matrix
\[ \begin{pmatrix}1&-2&0&0\\0&1&-1&0\\0&0&1&-2\\0&0&0&1\end{pmatrix}. \]
Thus $\Ga$ is the Grothendieck group of the path algebra of the quiver (see \S\ref{se:hereditary}).
\[ \cdot \mathrel{\vcenter{\offinterlineskip\vskip1.2pt
    \hbox{$\longrightarrow$}\vskip1pt\hbox{$\longrightarrow$}}} \cdot
\longrightarrow \cdot \mathrel{\vcenter{\offinterlineskip\vskip1.2pt
    \hbox{$\longrightarrow$}\vskip1pt\hbox{$\longrightarrow$}}} \cdot \]
We take $E=(e_1,\ldots,e_4)$ to be the standard basis in order. Then the element $a=(1,1,3,1)$ is a pseudo-real root but not a real root.
\end{exm}

Let $(\Ga,E)$ be a generalised Cartan lattice with Cartan matrix $C$. We say that $C$ is \emph{indecomposable} if we cannot permute the rows and columns to obtain a block diagonal matrix. When $C$ is indecomposable we say that $C$ is of
\begin{enumerate}
\item \emph{Dynkin type} if it is positive-definite, which is if and only if there exists $a>0$ such that $Ca>0$.
\item \emph{affine type} if it is positive semi-definite but not positive-definite, which is if and only if there exists $a>0$ such that $Ca=0$.
\item \emph{indefinite type} otherwise, which is if and only if there exists $a>0$ such that $Ca<0$.
\item \emph{hyperbolic type} if it is indefinite, but all its proper principal submatrices\footnote{
A \emph{principal submatrix} is obtained by deleting a set of columns and the matching rows.
}
are of Dynkin or affine type.
\end{enumerate}
See for example \cite[Sections 4 and 5]{Ka1990}. 

\begin{thm}\label{th:root}
Let $(\Ga,E)$ be a generalised Cartan lattice. If $(\Ga,E)$ is of Dynkin, affine or hyperbolic type, or else has rank two, then every pseudo-real root is in fact a real root.
\end{thm}

\begin{proof}
We first show that every pseudo-real root is either positive or negative. For Dynkin, affine or hyperbolic types this is \cite[Lemma~5.10~(b)]{Ka1990}, so suppose that $\Gamma$ has rank two and write $E=(e,f)$. Set $a:=\langle e,e\rangle$, $b:=\langle f,f\rangle$ and $c:=-\langle e,f\rangle$. Let $x=me-nf$ be a pseudo-real root where $m,n$ are integers having the same sign. Then $d:=\langle x,x\rangle=am^2+bn^2+cmn$ is a sum of three positive integers. On the other hand, $d$ divides both $am=\langle x,e\rangle$ and $bn=\langle f,x\rangle$, yielding a contradiction.

The proof now follows as for \cite[Proposition~5.10~(b)]{Ka1990}. Explicitly, let $x>0$ be a pseudo-real root of minimal height in its $W$ orbit. Since $\langle x,x\rangle>0$ we have $s_e(x)<x$ for some $e\in E$, so by minimality $s_e(x)<0$. Hence $x=me$ for some $m$, and then necessarily $x=e$.
\end{proof}

\subsection*{Generalised Cartan lattices of Dynkin type}

We finish this section by discussing the special case when $(\Gamma,E)$ is of Dynkin type.

Recall that every Coxeter system $(W,E)$ gives rise to a symmetric bilinear form (see Appendix B).

\begin{thm}[\cite{Hum1990}, Theorem 6.4]\label{th:finite}
Let $(W,S)$ be a Coxeter system. Then $W$ is finite if and only the corresponding symmetric bilinear form is positive definite. \qed
\end{thm}

In the finite case we also have the following fundamental result giving a geometric interpretation of the absolute length.

\begin{lem}[Carter's Lemma \cite{Ca1972}]\label{le:Carter}
Let $(W,S)$ be a Coxeter system, and suppose that $W$ is finite. Then $\ell(w)=\dim\Im(\id-w)$. Moreover, $w=s_{a_1}\cdots s_{a_r}$ is reduced if and only if the roots $a_i$ are linearly independent.
\end{lem}

\begin{proof}
Since $W$ is finite, the symmetric bilinear form is positive definite.

By induction on $r$ we see that, if $w=s_{a_1}\cdots s_{a_r}$ as a product of reflections, then
\[ \id-w = \sum_i(\lambda_i,-)a_i \quad\textrm{where } \lambda_i := 2s_{a_r}\cdots s_{a_{i+1}}(a_i)/(a_i,a_i). \]
In particular, $\Im(\id-w)\subseteq\mathrm{Span}(a_i)$, and so $\dim\Im(\id-w)\le\ell(w)$.

Conversely, observe that $\Im(\id-w)$ is the orthogonal complement to $\mathrm{Fix}(w)=\mathrm{Fix}(w^{-1})$. Moreover, by \cite[Theorem 1.1(d)]{Hum1990} we know that $w$ can be written as a product of reflections $s_a$ for roots $a\in\Im(\id-w)$. In particular, if $w\neq\id$, then there exists some root $a\in\Im(\id-w)$. Write $a=x-w(x)$. Then $(x,x)=(w(x),w(x))$ implies $2(a,x)=(a,a)$, and hence $s_a(x)=x-a=w(x)$. Thus $s_aw(x)$ fixes everything in $\mathrm{Fix}(w)$ as well as $x$. By induction on dimension we deduce that $w$ can be written as a product of at most $\dim\Im(\id-w)$ reflections, so $\ell(w)\le\dim\Im(\id-w)$.

This proves that $\ell(w)=\dim\Im(\id-w)$.

Next suppose that $w=s_{a_1}\cdots s_{a_r}$ with the $a_i$ linearly independent. Then so too are the $\lambda_i$ above, and hence we can find elements $x_i\in V$ such that $(\lambda_i,x_j)=\delta_{ij}$. It follows that $(\id-w)(x_i)=a_i$, and so $\dim\Im(\id-w)=r$. Thus $\ell(w)=r$ and the expression is reduced.

On the other hand, if the expression $w=s_{a_1}\cdots s_{a_r}$ is reduced, then $r=\dim\Im(\id-w)\le\dim\mathrm{Span}(a_i)$, and so the $a_i$ must be linearly independent.
\end{proof}

Finally, we have another characterisation of Dynkin type. Again, this is true for all Coxeter groups, but we offer a simple proof for Weyl groups exhibiting the usefulness of generalised Cartan lattices (cf.\ \cite{Ho1982}).

\begin{thm}\label{th:Dynkin-type}
Let $(\Gamma,E)$ be a generalised Cartan lattice, with Weyl group $W$ and Coxeter element $c$. Then $(\Gamma,E)$ is of Dynkin type if and only if $T$ is finite, if and only if $\NC(\Gamma,E)$ is finite, if and only if $c$ has finite order.
\end{thm}

\begin{proof}
We know from Theorem~\ref{th:finite} that $(\Gamma,E)$ is of Dynkin type if and only if $W$ is finite. Also, it is clear that $W$ finite implies $T$ is finite, which in turn implies that $\NC(W,c)$ is finite. Next, if $\NC(W,c)$ is finite, then since every reflection $s_a$ for $a\in\{c^r(e_i)\}$ is a non-crossing partition, $c$ must have finite order on each $e_i$, so $c$ has finite order in $W$. Finally, assume that $c$ has finite order $h$. Write $c=s_1s_2\cdots s_n$, and set
\[ p_i := s_n\cdots s_{i+1}(e_i)  \quad\textrm{and}\quad q_i := s_1\cdots s_{i-1}(e_i) = -c(p_i). \]
We note that if $a\in\Phi_+$, then $c(a)<0$ if and only if $a=p_i$ for some $i$, and $c^{-1}(a)<0$ if and only if $a=q_i$ for some $i$.

If $T$ is infinite, then there exists some $a\in\Phi^+$ not of the form $c^r(p_i)$ or $c^r(q_i)$. It follows that $\delta:=a+c(a)+\cdots+c^{h-1}(a)>0$ is $c$-invariant. Then necessarily $s_i(\delta)=\delta$ for all $i$, so $\delta\in\rad\Gamma$. Now, using Proposition~\ref{pr:cox} we get
\[ \langle\delta,c(x)\rangle = -\langle x,\delta\rangle = \langle\delta,x\rangle. \]
Also, by induction we have $r_i\cdots r_1(\delta) = \sum_{j>i}\delta_je_j$, and hence
\[ \langle\delta,p_i\rangle<0 \quad\textrm{and}\quad \langle\delta,q_i\rangle>0. \]
It follows that $c^r(p_i)>0$ for all $r\leq0$, so that $c$ has infinite order, a contradiction. Thus $T$ must be finite, and hence $W$ is finite (as in \cite[Exercise 5.6 (2)]{Hum1990}).
\end{proof}

\section{Grothendieck groups of hereditary algebras}\label{se:hereditary}

The Grothendieck group of a finite dimensional hereditary algebra is an example of a generalised Cartan lattice, and in fact each generalised Cartan lattice is of this form (Lemma~\ref{le:hereditary}). In this section we concentrate on exceptional sequences of modules over hereditary algebras and discuss the braid group action.

\subsection*{Hereditary algebras}

Let $k$ be a field and $A$ a finite dimensional $k$-algebra. We denote by $\mod A$ the category of finite dimensional $A$-modules and by $\proj A$ the full subcategory consisting of projective $A$-modules.

The \emph{Grothendieck group} $K_0(A)$ is by definition the Grothendieck group of the exact category $\proj A$ with the bilinear form given by
\[ \langle [X],[Y]\rangle := \dim_k\Hom_A(X,Y). \]
This group is free of finite rank, with basis the classes of the indecomposable projective $A$-modules.

Now suppose that $A$ has finite global dimension. Let $K_0(\mod A)$ denote the Grothendieck group of the abelian category $\mod A$ and observe that the inclusion $\proj A\to \mod A$ induces an isomorphism $K_0(A)\xto{\sim}K_0(\mod A)$ which identifies the bilinear form on $K_0(A)$ with the \emph{Euler form} on $K_0(\mod A)$, given as
\[ \langle [X],[Y]\rangle := \sum_{i\ge 0}(-1)^i\dim_k\Ext^i_A(X,Y). \] 
We view this isometry as an identification.

Finally, an algebra $A$ is called \emph{hereditary} if each $A$-module has a projective resolution of length at most one. This is equivalent to saying that every submodule of a projective module is again projective.

The following lemma shows that every generalised Cartan lattice can be realised as the Grothendieck group of a finite dimensional hereditary algebra.

\begin{lem}\label{le:hereditary}
The assignment $A\mapsto K_0(A)$ has the following properties:
\begin{enumerate}
\item Let $A$ be a finite dimensional hereditary algebra. Then we can order a complete set of representatives for the simple $A$-modules as $S_1,\ldots,S_n$ such that, setting $e_i:=[S_i]$ and $E:=(e_1,\ldots,e_n)$, then $(K_0(A),E)$ is a generalised Cartan lattice.
\item Let $(\Ga,E)$ be a generalised Cartan lattice, where $E=(e_1,\ldots,e_n)$. Given a finite field $k$ there exists a finite dimensional hereditary $k$-algebra $A$ and an isometry $\Ga\xto{\sim}K_0(A)$ sending each $e_i$ to the class of a simple $A$-module.
\end{enumerate}
\end{lem} 

\begin{proof}
(1) Note first that $[S_1],\ldots,[S_n]$ form a basis for $K_0(A)$. Also, each $\End_A(S_i)$ is a division algebra and each $\Ext^1_A(S_i,S_j)$ is an $\End_A(S_j)$-$\End_A(S_i)$-bimodule. In particular, $\langle [S_i],[S_i]\rangle$ divides both $\langle [S_i],[S_j]\rangle$ and $\langle[S_j],[S_i]\rangle$ for all $j$. Thus each $[S_i]$ is a pseudo-real root.

Now, each non-zero morphism between indecomposable projective $A$-modules is a monomorphism since $A$ is hereditary. Thus each indecomposable projective is exceptional (since $A$ is finite dimensional) and we can order a representative set of indecomposable projective modules as $P_1,\ldots,P_n$ such that $\Hom_A(P_i,P_j)=0$ for $i<j$. Let $S_i=P_i/\rad P_i$ be the simple top of the projective $P_i$. Then the long exact sequence for $\Hom_A(-,S_j)$ yields
\[ 0 \to \Hom_A(S_i,S_j) \to \Hom_A(P_i,S_j) \to \Hom_A(\rad P_i,S_j) \to \Ext^1_A(S_i,S_j) \to 0. \]
Since $\Hom_A(\rad P_i,S_j)=0$ for all $j\le i$ it follows that each $S_i$ is exceptional, and that $([S_1],\ldots,[S_n])$ is a complete, orthogonal exceptional sequence.

(2) We follow \cite[Section~7]{Ga1973} and \cite[Section~5]{Hu}. Let $k_i/k$ be a field extension of degree $\langle e_i,e_i\rangle$ and $k_{ij}/k$ a
field extension of degree $-\langle e_i,e_j\rangle$ for $i<j$. Set $k_{ij}=0$ for $i\ge j$. We regard each $k_{ij}$ as $k_j$-$k_i$-bimodule. Then $A_0=\prod_i k_i$ is a semisimple $k$-algebra and $A_1=\bigoplus_{i,j}k_{ij}$ an $A_0$-bimodule, so the tensor algebra
\[ A := \bigoplus_{p\ge 0}A_p \quad\textrm{where}\quad A_p:=A_1\otimes_{A_0}\cdots\otimes_{A_0}A_1 \quad\textrm{($p$
  times)} \]
is a finite dimensional hereditary $k$-algebra. Denote by $\e_i$ the idempotent of $A$ given by the identity of $k_i$. Then the $P_i:=A\e_i$ give a representative set of indecomposable projective $A$-modules, with simple tops $S_i=k_i$, and $\e_j(\rad P_i/\rad^2
P_i)\cong k_{ij}$. Thus
\[ \End_A(S_i)\cong k_i \quad\textrm{and}\quad \Ext^1_A(S_i,S_j) \cong \Hom_A(\rad P_i,S_j) \cong \Hom_{k_j}(k_{ij},k_j), \]
see for example \cite[Proposition~2.4.3]{Be1991}, so that $\langle[S_i],[S_j]\rangle=\langle e_i,e_j\rangle$ as required.
\end{proof}

Given a finite dimensional hereditary algebra $A$, we will abuse notation and just write $K_0(A)$ for the corresponding generalised Cartan lattice with the natural choice of a complete orthogonal exceptional sequence given by the simple $A$-modules.

\subsection*{Exceptional sequences}

Let $A$ be a finite dimensional hereditary algebra. A module $X\in\mod A$ is called \emph{exceptional} if it is indecomposable and $\Ext_A^1(X,X)=0$. A sequence $(X_1,\ldots,X_r)$ of finite dimensional $A$-modules is called \emph{exceptional} if each $X_i$ is exceptional and $\Hom_A(X_i,X_j)=0=\Ext_A^1(X_i,X_j)$ for all $i>j$. Such a sequence is \emph{complete} if $r$ equals the rank of $K_0(A)$, and is \emph{orthogonal} if $\Hom_A(X_i,X_j)=0$ for all $i\neq j$. For example, any exceptional sequence consisting of simples is necessarily orthogonal.

We begin by recalling the following useful lemmas of Happel and Ringel, and Kerner.

\begin{lem}[{\cite[Lemma~4.1]{HR1982}}]\label{le:mono-or-epi}
Let $X$ and $Y$ be indecomposable modules. If $\Ext^1_A(X,Y)=0$, then any homomorphism $Y\to X$ is either mono or epi. In particular, if $X$ is exceptional, then $\End_A(X)$ is a division algebra. \qed
\end{lem}

\begin{lem}[{\cite[Lemma~8.2]{Ke1996}}]\label{le:rigid-dim-vector}
Let $X$ and $Y$ be rigid modules, so $\Ext^1_A(X,X)=0=\Ext^1_A(Y,Y)$. If $[X]=[Y]$ in $K_0(A)$, then $X\cong Y$. \qed
\end{lem}

\begin{prop}\label{pr:exceptional-module-real-root}
Let $X$ be exceptional. Then $[X]\in K_0(A)$ is a real root.
\end{prop}

\begin{proof}
We first note that the result holds when $K_0(A)$ has rank two, using Theorem~\ref{th:root} (cf.\ \cite[Section 3]{Ri1976}).

In general, let $X$ be a non-simple exceptional module. By Schofield's result, Proposition~\ref{pr:Schofield}, we can find an orthogonal exceptional pair $(U,V)$ such that $X\in\C(U,V)$ is not simple. Since $[U],[V]<[X]$, we know by induction that $[U]$ and $[V]$ are real roots. Moreover, as above, $[X]$ is obtained from either $[U]$ or $[V]$ by applying an element of the subgroup $\langle s_{[U]},s_{[V]}\rangle\leq W$. Hence $[X]$ is also a real root.
\end{proof}

It follows that each exceptional sequence $(X_1,\ldots,X_r)$ in $\mod A$ yields an exceptional sequence $([X_1],\ldots,[X_r])$ in $K_0(A)$. Moreover, if the former is complete (respectively orthogonal), then so too is the latter.

The real roots of the form $\pm [X]$ for an exceptional object $X$ are called \emph{real Schur roots}. If the algebra $A$ is of finite representation type (so the corresponding Weyl group $W(K_0(A))$ is finite), or if $\rk K_0(A)=2$, then all real roots are real Schur roots. The following example exhibits a real root which is not a real Schur root.

\begin{exm}
Consider the path algebra of the following quiver
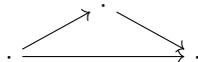
\[ \begin{tikzcd}[row sep=tiny]
&\mathbf{\cdot} \arrow{dr}\\
\mathbf{\cdot} \arrow{ur}\arrow{rr} && \mathbf{\cdot}
\end{tikzcd} \]
There is a unique indecomposable module $X$ with dimension vector
\[ [X]=\begin{matrix}\;\;\;2\;\;\;\\[-0.5em]1\;\;\;\;\;\;1\end{matrix} \]
Then $[X]$ is a real root, but the module $X$ is not exceptional.
\end{exm}

\subsection*{The braid group action}

Let $X=(X_1,\ldots,X_r)$ be an exceptional sequence in $\mod A$. We define $\C(X)$ to be the smallest full subcategory of $\mod A$ containing each $X_i$ and closed under kernels, cokernels and extensions. Then $\C(X)$ is equivalent to the module category of a finite dimensional hereditary algebra by Theorem~\ref{th:subcat}. Also, by Corollary~\ref{co:extend}, for each integer $1\leq i<r$ there exist unique modules $R_{X_{i+1}}(X_i)$ and $L_{X_i}(X_{i+1})$ in $\C(X)$ yielding exceptional sequences
\begin{gather*}
(X_1,\ldots,X_{i-1},X_{i+1},R_{X_{i+1}}(X_i),X_{i+2},\ldots,X_r)\\
(X_1,\ldots,X_{i-1},L_{X_i}(X_{i+1}),X_i,X_{i+2},\ldots,X_r).
\end{gather*}

Following \cite{CB1992,Ri1994}, the braid group $B_r$ acts on exceptional sequences of length $r$ via
\begin{gather*}
\s_i(X_1,\ldots,X_r) := (X_1,\ldots,X_{i-1},X_{i+1},R_{X_{i+1}}(X_i),X_{i+2},\ldots,X_r)\\
\s_i^{-1}(X_1,\ldots,X_r) := (X_1,\ldots,X_{i-1},L_{X_i}(X_{i+1}),X_i,X_{i+2},\ldots,X_r).
\end{gather*}

In fact, we can describe the modules $L_X(Y)$ and $R_Y(X)$ explicitly, using the five term exact sequences \eqref{eq:perp} and \eqref{eq:perp2}, together with Remark~\ref{re:perp}. Let $(X,Y)$ be an exceptional pair. If $\Hom_A(X,Y)=0$, then $L_X(Y)$ is the middle term of the universal extension
\[ 0\lto Y\lto L_X(Y)\lto \Ext^1_A(X,Y)\otimes_{\End_A(X)} X\lto 0. \]
Otherwise, if $\Hom_A(X,Y)\neq0$, then every morphism is either mono or epi by Lemma~\ref{le:mono-or-epi}. In this case the canonical
morphism $\Hom_A(X,Y)\otimes_{\End_A(X)} X\to Y$ is also mono or epi, by \cite[Lemma 3.1]{RS1990}, and we define $L_X(Y)$ to be its cokernel or kernel, respectively. Thus $L_X(Y)$ is given by one of the following exact sequences
\begin{gather*}
0\lto \Hom_A(X,Y)\otimes_{\End_A(X)} X\stackrel{\can}\lto Y\lto L_X(Y)\lto 0\phantom{.}\\
0\lto L_X(Y)\lto \Hom_A(X,Y)\otimes_{\End_A(X)} X\stackrel{\can}\lto Y\lto 0.
\end{gather*}
An analogous description is used for $R_Y(X)$.

We observe that this definition appears more natural once one passes to the derived category $\D^b(\mod A)$, where functors $\mathcal L_E$ and $\mathcal R_E$ are defined with respect to any exceptional object $E$; see \cite{Bo1990}. Then $L_X(Y)$ and $R_Y(X)$ coincide up to translation with the objects defined in $\D^b(\mod A)$ via the functors $\mathcal L_X$ and $\mathcal R_Y$.

\subsection*{Connection to real exceptional sequences}

We now want to compare, for a finite dimensional hereditary algebra $A$, exceptional sequences in $\mod A$, real exceptional sequences in $K_0(A)$, and non-crossing partitions in $W(K_0(A))$.

We say that two exceptional sequences $(X_1,\ldots,X_r)$ and $(Y_1,\ldots,Y_r)$ in $\mod A$ are \emph{isomorphic} provided $X_i\cong Y_i$ for all $i$.

\begin{prop}\label{pr:exceptional-ncp}
The maps
\[ (X_1,\ldots,X_r) \mapsto ([X_1],\ldots,[X_r]) \quad\textrm{and}\quad (f_1,\ldots,f_r) \mapsto (s_{f_1},\ldots,s_{f_r}) \]
yield $B_r$-equivariant bijections between
\begin{enumerate}
\item isomorphism classes of exceptional sequences of length $r$ in $\mod A$,
\item real exceptional sequences of length $r$ in $K_0(A)$, up to the action of the sign group, and
\item sequences of reflections $(t_1,\ldots,t_r)$ in the Weyl group $W(K_0(A))$ such that $w=t_1\cdots t_r$ is a non-crossing partition of absolute length $r$.
\end{enumerate}
\end{prop}

\begin{proof}
Let $(X_1,\ldots,X_r)$ be an exceptional sequence. By Corollary~\ref{co:extend} we can extend to a complete exceptional sequence $(X_1,\ldots,X_n)$. As each $[X_i]$ is a real root by Proposition~\ref{pr:exceptional-module-real-root}, $([X_1],\ldots,[X_n])$ is a complete real exceptional sequence in $K_0(A)$, and hence $([X_1],\ldots,[X_r])$ is a real exceptional sequence of length $r$. Moreover, this map yields an injection, since if $X$ and $Y$ are exceptional modules such that $[X]=[Y]$, then $X\cong Y$ by Lemma~\ref{le:rigid-dim-vector}.

To see that this map is $B_r$-equivariant, it is enough to check it for the generators $\sigma_i^{\pm1}$, and hence just for $r=2$. Let $(X,Y)$ be an exceptional pair in $\mod A$. If $\Hom_A(X,Y)=0$, or if there is a monomorphism $X\hookrightarrow Y$, then the construction of $L_X(Y)$ yields $[L_X(Y)]=s_{[X]}([Y])$. Otherwise, there is an epimorphism $X\twoheadrightarrow Y$ and $[L_X(Y)]=-s_{[X]}([Y])$. This proves the result for $\sigma_i^{-1}$. The proof for $\sigma_i$ is analogous.

Let $F=(f_1,\ldots,f_r)$ be a real exceptional sequence. Then $s_F=s_{f_1}\cdots s_{f_r}$ is a non-crossing partition of absolute length $r$ by Proposition~\ref{pr:real-seq}, and so $(s_{f_1},\ldots,s_{f_r})$ has the required properties. It is $B_r$-equivariant by Lemma \ref{le:equivar}. Moreover, this map also yields an injection, since if $e,f$ are real roots such that $s_e=s_f$, then $e=\pm f$ by Lemma~\ref{le:preservation-of-Coxeter}.

Finally, as in Lemma~\ref{le:hereditary} let $S=(S_1,\ldots,S_n)$ be a complete, othogonal exceptional sequence in $\mod A$ consisting of simple modules. Set $e_i:=[S_i]$ and $s_i=s_{e_i}$, so that $E:=(e_1,\ldots,e_n)$ is a complete, orthogonal exceptional sequence in $K_0(A)$ and the Coxeter element is $c=s_1\cdots s_n$.

Now let $(t_1,\ldots,t_r)$ be a sequence of reflections in $W$ such that $t_1\cdots t_r$ is a non-crossing partition of absolute length $r$. Write the Coxeter element as $c=t_1\cdots t_n$. Then by Theorem~\ref{th:IS2010} we have $(t_1,\ldots,t_n)=\sigma(s_1,\ldots,s_n)$ for some $\sigma\in B_n$. Set $(X_1,\ldots,X_n):=\sigma(S_1,\ldots,S_n)$. It follows that $(X_1,\ldots,X_r)$ is an exceptional sequence and $s_{[X_i]}=t_i$ for all $i$. Thus the composition of the three maps is the identity, so they are all bijections.
\end{proof}

Using this proposition we get an alternative proof of the following transitivity result.

\begin{thm}[Crawley-Boevey \cite{CB1992}, Ringel \cite{Ri1994}]\label{th:braid}
Let $A$ be a finite dimensional hereditary algebra, and set $n$ to be the rank of $K_0(A)$. Then the braid group $B_n$ acts transitively on the isomorphism classes of complete exceptional sequences in $\mod A$.
\end{thm}

\begin{proof}
The bijection between isomorphism classes of complete exceptional sequences and factorisations of the Coxeter element is equivariant for the action of the braid group. Since the action on factorisations of the Coxeter element is transitive by Theorem~\ref{th:IS2010}, so too is the action on complete exceptional sequences.
\end{proof}

We can also use Proposition~\ref{pr:exceptional-ncp} to characterise the real Schur roots amongst all real roots using non-crossing partitions.

\begin{cor}\label{co:braid2}
Let $a\in K_0(A)$ be a real root. Then $a$ is a real Schur root if and only if $s_a\in\NC(K_0(A))$. In particular, this depends only on $K_0(A)$, and not on the algebra $A$ itself. \qed
\end{cor}

We finish by observing that Proposition~\ref{pr:exceptional-ncp} can be reformulated as saying exceptional sequences in $\mod A$ correspond to paths in the Hasse diagram of $\NC(K_0(A))$. For, sequences of reflections $(t_1,\ldots,t_r)$ in $W(K_0(A))$ with $w=t_1\cdots t_r$ a non-crossing partition of absolute length $r$ are represented by paths of length $r$ which start at the unique minimal element; see also \cite[p.~1534]{IT2009}. The number of complete exceptional sequences for algebras of finite representation type is computed in \cite{ONSFR2013}; we refer to their paper for further historical comments.

\subsection*{Application to Gabriel's Theorem}

We can use the results obtained so far to give a root-theoretic proof of Gabriel's Theorem, so in particular not requiring the development of Auslander-Reiten theory, or even reflection functors (cf.\ \cite{DR1974,Ri1983}). As such, this answers the question posed by Gabriel in \cite[Section~4]{Ga1973}, but now for all Dynkin types, not just $ADE$-type.

\begin{thm}\label{th:Gabriel}
Let $A$ be a finite dimensional hereditary algebra, either of Dynkin type or of rank two. Then the map $X\mapsto[X]$ induces a bijection between the isomorphism classes of exceptional modules and the positive real roots.

Moreover, $A$ is representation-finite if and only if it is of Dynkin type, which is if and only if every indecomposable module is exceptional.
\end{thm}

\begin{proof}
Corollary~\ref{co:braid2} tells us that the map $X\mapsto s_{[X]}$ induces a bijection between the isomorphism classes of exceptional $A$-modules and those reflections which are non-crossing partitions. For the first part it is therefore enough to show that when $W$ is finite or has rank two, then every reflection is a non-crossing partition.

Suppose first that $W$ is finite, and let $t\in T$ be any reflection. By Carter's Lemma~\ref{le:Carter} we have $\ell(tc)\le n=\ell(c)$, and since we cannot have equality we must have $\ell(tc)<n$. Thus $t<c$ is a non-crossing partition.

Now suppose that $W$ has rank two, say with simple reflections $s,t$. The only two Coxeter elements are $c=st$ and $c^{-1}=ts$. Clearly every reflection in $W$ is of the form $c^rsc^{-r}$ or $c^rtc^{-r}$ for some $r\in\mathbb Z$, so every reflection lies in both $\NC(W,c)$ and $\NC(W,c^{-1})$.

For the second part, if $A$ is representation-finite, then there are only finitely many exceptional modules up to isomorphism, so only finitely many non-crossing partitions. By Theorem~\ref{th:Dynkin-type} this implies that $W$ is finite, so $A$ is of Dynkin type.

Assume next that every indecomposable module is exceptional. Given $a>0$ in $K_0(A)$, let $X$ be an $A$-module with $[X]=a$ and $\dim\End_A(X)$ minimal. Then $\Ext^1_A(X,X)=0$. For, this is clear if $X$ is indecomposable, so assume $X=Y\oplus Z$ with $\Ext^1(Y,Z)\neq0$. Then there is a non-split short exact sequence $0\to Z\to X'\to Y\to 0$, in which case $[X']=a$ and $\dim\End_A(X')<\dim\End_A(Y\oplus Z)$, a contradiction. It follows that $\langle a,a\rangle=\langle[X],[X]\rangle>0$, so $A$ is of Dynkin type using the classification of generalised Cartan lattices given before Theorem~\ref{th:root}. Since there are now only finitely many non-crossing partitions, we also see that $A$ is representation-finite.

Finally, if $A$ is of Dynkin type, then the lemma below tells us that every indecomposable module is exceptional, completing the proof.
\end{proof}

We call an $A$-module $X$ a \emph{brick} provided $\End_A(X)$ is a division algebra.

\begin{lem}[\cite{Ri1983}]
Let $A$ be a finite dimensional hereditary algebra. If there is an indecomposable $A$-module $X$ which is not exceptional, then there is a brick $X'\subseteq X$ which is not exceptional. In particular, $\langle[X'],[X']\rangle\leq0$, so that $A$ is not of Dynkin type.
\end{lem}

\begin{proof}
If $X$ is not a brick, then take an endomorphism $f$ such that $I:=\Im(f)$ has minimal dimension. Note that $\Ext^1_A(I,M)\neq0$ for every indecomposable summand $M$ of $\Ker(f)$. Next, by minimality, $I$ is a brick. The image of $\Hom_A(I,X)$ in $\End_A(I)$ is a proper left ideal (as the sequence is not split), and so must be zero. Thus $\Hom_A(I,\Ker(f))\cong\Hom_A(I,X)$ is non-zero. Take an indecomposable summand $X_1$ of $\Ker(f)$ such that there is a non-zero map $I\to X_1$. Again, by minimality, this must be injective, and hence we have an epimorphism $\Ext^1_A(X_1,X_1)\twoheadrightarrow\Ext^1_A(I,X_1)$, so that $X_1$ is not exceptional. The result now follows by induction on $\dim X$.
\end{proof}

\section{The subobjects of generalised Cartan lattices}\label{se:subobjects}

In this section we establish the correspondence between subobjects of generalised Cartan lattices and non-crossing partitions. Later on we illustrate this by looking at representations of hereditary algebras. In fact, the proof of our main result uses representations of hereditary algebras in an essential way (see Remark~\ref{re:need-hered}).

\subsection*{Subobjects}
Fix a category and an object $X$ in this category. Two monomorphisms $\p\colon U\to X$ and $\p'\colon U'\to X$ are \emph{equivalent} if $\p$ factors through $\p'$ and $\p'$ factors through $\p$. The equivalence class of a monomorphism $\p$ ending in $X$ is denoted by $[\p]$; they form the \emph{subobjects} of $X$ and $\Sub(X)$ denotes the set of these subobjects. Defining $[\p]\le[\p']$ if $\p$ factors through $\p'$ gives a partial order on $\Sub(X)$.

The following lemma provides a crucial step in our proof of the main theorem. It would be interesting to have a purely combinatorial proof which avoids the use of hereditary algebras.

\begin{lem}\label{le:orbits}
Let $(\Ga,E)$ be a generalised Cartan lattice, and $F$ a real exceptional sequence of length $r$. Then
\begin{enumerate}
\item $(\bbZ F,\sigma F)$ is a generalised Cartan lattice for some $\sigma$, and
\item if $F'$ is any other real exceptional sequence such that $\bbZ F'=\bbZ F$, then there exists $\sigma\in\{\pm1\}\wr B_r$ such that $F'=\sigma F$.
\end{enumerate}
\end{lem}

\begin{proof}
Let $E=(e_1,\ldots,e_n)$. By Lemma~\ref{le:hereditary} we can find a finite dimensional hereditary $k$-algebra $A$ and simple $A$-modules $S_i$ such that $[S_i]=e_i$.

(1) Write $F=(f_1,\ldots,f_r)$. By Proposition~\ref{pr:exceptional-ncp} we can find an exceptional sequence $X=(X_1,\ldots,X_r)$ in $\mod A$ such that $F=\e([X_1],\ldots,[X_r])$ for some $\e\in\{\pm1\}^r$. Then $\C(X)\cong\mod B$ for some finite dimensional hereditary algebra, by Theorem~\ref{th:subcat}, so $\bbZ F=K_0(\C(X))$ is naturally a generalised Cartan lattice.

(2) Begin by extending $X$ to a complete exceptional sequence $(X,Y)$ for some $Y=(Y_1,\ldots,Y_s)$. Set $g_i:=[Y_i]$ and $G=(g_1,\ldots,g_s)$. Then $(F,G)$ is a complete real exceptional sequence in $\Gamma$.

Now let $F'=(f'_1,\ldots,f'_r)$ be another real exceptional sequence such that $\bbZ F'=\bbZ F$. Again, we can lift this modulo the sign group action to an exceptional sequence $X'=(X'_1,\ldots,X'_r)$. Note that $\bbZ G={}^\perp F={}^\perp F'$, so that $(F',G)$ is also a complete real exceptional sequence in $\Gamma$. It follows from Proposition~\ref{pr:exceptional-ncp} that $(X',Y)$ is also complete exceptional sequence in $\mod A$. Therefore $\C(X)=\C(Y)^\perp=\C(X')$, and so $X'=\sigma(X)$ for some $\sigma\in B_r$ by Theorem~\ref{th:braid}. Hence $F'$ and $\sigma(F)$ agree up to signs.
\end{proof}

Let $\Exc(\Ga,E)$ be the set of equivalence classes of real exceptional sequences in $(\Ga,E)$, where two such sequences of length $r$ are equivalent if they determine the same orbit under the action of $\{\pm1\}\wr B_r$. We make this into a poset by saying that $[G]\le[F]$ if $G$ is an initial subsequence of some $\sigma F$.

\begin{thm}\label{th:main}
Let $(\Ga,E)$ be a generalised Cartan lattice. Then there are canonical poset isomorphisms
\[ \Sub(\Ga,E) \stackrel{\sim}\lto \Exc(\Ga,E) \stackrel{\sim}\lto \NC(\Ga,E). \]
The first sends the class of a monomorphism $\p\colon(\Ga',E')\to(\Ga,E)$ to $[\p E']$, and the second sends $[F]$ to $s_F$.
\end{thm}

\begin{proof}
By Proposition~\ref{pr:real-seq} the map $\Exc(\Ga,E)\to \NC(\Ga,E)$, $F\mapsto s_F$, is well-defined and surjective. Moreover, $s_F=s_{F'}$ if and only if $\bbZ F=\bbZ F'$, which by Lemma~\ref{le:orbits}~(2) occurs if and only if $[F']=[F]$. To see that it is an isomorphism of posets, note that $s_{F'}\le s_F$ if and only if $s_F=s_{F'}s_{F''}$ for some $F''$; since we have a bijection, this is if and only if $(F',F'')=\sigma F$ for some $\sigma$, or equivalently $[F']\leq[F]$.

Now consider the map $\Sub(\Ga,E)\to\Exc(\Ga,E)$. Given monomorphisms of generalised Cartan lattices $(\Ga'',E'')\xrightarrow{\psi}(\Ga',E')\xrightarrow{\phi}(\Ga,E)$, we know that $\psi E''$ is a real exceptional sequence in $(\Ga,E')$, so is an initial subsequence of some $\sigma E'$. Thus $\phi\psi E''$ is an initial subsequence of $\sigma\phi E'$, and hence that $[\phi\psi E'']\leq[\phi E']$. This shows that we have a map of posets.

Conversely, note that if $\sigma E$ is again orthogonal, then the identity map on $\Ga$ yields an isomorphism of generalised Cartan lattices $(\Ga,E)\cong(\Ga,\sigma E)$. By Lemma~\ref{le:orbits}~(1) we therefore have a well-defined map $\Exc(\Ga,E)\to\Sub(\Ga,E)$ sending $[F]$ to the class of the inclusion $(\bbZ F,\sigma F)\hookrightarrow(\Ga,E)$, and this map is inverse to the one above. Moreover, if $[F']\leq[F]$, where for simplicity $F'$ and $F$ are both orthogonal, then the inclusion $(\bbZ F',F')\hookrightarrow(\Ga,E)$ factors through $(\bbZ F,F)\hookrightarrow(\Ga,E)$. Thus $\Exc(\Ga,E)\to\Sub(\Ga,E)$ is also a map of posets.
\end{proof}

We write $\cox$ for the map $\Exc(\Ga)\to\NC(\Ga)$, $E'\mapsto s_{E'}$, as well as for the composition $\Sub(\Ga)\to\NC(\Ga)$.

\begin{cor}\label{co:map-on-nc}
Each morphism $\p\colon(\Ga',E')\to(\Ga,E)$ of generalised Cartan lattices induces a commutative diagram
\[\begin{CD}
\Sub(\Ga',E') @>{\sim}>> \Exc(\Ga',E') @>{\sim}>> \NC(\Ga',E')\\
@VVV @VVV @VVV\\
\Sub(\Ga,E) @>{\sim}>> \Exc(\Ga,E) @>{\sim}>> \NC(\Ga,E)
\end{CD} \]
It follows that $(\Ga,E)\mapsto\NC(\Ga,E)$ determines a functor from the category of generalised Cartan lattices to the category of posets.
\end{cor}

\begin{proof}
We know that $\p$ sends real exceptional sequences in $(\Ga',E')$ to real exceptional sequences in $(\Ga,E)$. This preserves the action of the wreath product and the notion of being an initial subsequence, so induces a morphism of posets $\Exc(\Ga',E')\to\Exc(\Ga,E)$. Since $[\psi]\in\Sub(\Ga',E')$ is sent to $[\p\psi]\in\Sub(\Ga,E)$, the square on the left is commutative.
\end{proof}

\subsection*{The Weyl group}

We now show that the assignment $(\Ga,E)\mapsto W(\Ga,E)$ is also functorial, and restricts to the functor $(\Ga,E)\mapsto\NC(\Ga,E)$ constructed above. The proof requires several steps.

\begin{lem}\label{le:weyl}
Let $(\Gamma,E)$ be an indecomposable generalised Cartan lattice of rank $n+1$ and with Weyl group $W:=W(\Gamma,E)$. Let $\Gamma'\leq\Gamma$ be a sublattice with basis $\{f_1,\ldots,f_n\}$ consisting of real roots and write $W'\leq W$ for the subgroup generated by $s_{f_1},\ldots,s_{f_n}$. Then the restriction map $W'\to\Aut(\Gamma')$ is injective.
\end{lem}

\begin{proof}
Suppose $w\in W'$ fixes each $f_i$. We need to show that $w$ is the identity. It will be convenient at times to extend scalars to the rationals, yielding the vector spaces $\mathbb Q\Gamma'\leq\mathbb Q\Gamma$.

Take $e\in E\setminus\Gamma'$ and set $x:=w(e)-e$. Assume for contradiction that $x\neq0$. We claim that $x\in\rad(\Gamma)$ and is either positive or negative.

Note first that $x\in\Gamma'$. Next, for all $f\in\Gamma'$ we have $w(f)=f$, so
\[ (e,f) = (w(e),w(f)) = (e+x,f) \]
and hence $(x,f)=0$ for all $f\in\Gamma'$. In particular, $(x,x)=0$. Also,
\[ (e,e) = (w(e),w(e)) = (e+x,e+x) = (e,e) + 2(e,x) \]
and hence $(e,x)=0$. Thus $x\in\rad(\Gamma)$.

Next write $x=x_+-x_-$ with $x_+,x_-\geq0$ having disjoint support. Note that $(x_+,e')=(x_-,e')$ for all $e'\in E$. If $e'$ is not in the support of $x_+$ then the the left hand side is non-positive, whereas if $e'$ is not in the support of $x_-$ then the right hand side is non-positive. We conclude that $(x_+,e')=(x_-,e')\leq0$ for all $e'\in E$.

Finally, $w(e)=e+x$ is a real root, so is either positive or negative. If $x_+>0$, then $x_-=0$ or $x_-=e$, and in the latter case $(x_-,e)>0$, a contradiction. Thus $x=x_+>0$. Otherwise, if $x_+=0$, then $x=-x_-<0$. This proves the claim.

By the description of the different types of Cartan matrices given before Theorem~\ref{th:root} we see that we have reduced to the case when $(\Gamma,E)$ is of affine type. In this case $\rad(\Gamma)=\mathbb Z\delta$ for some $\delta>0$, \cite[Theorem~5.6~(b)]{Ka1990}, so $x=m\delta$ with $m\neq0$. In particular, $\delta\in\mathbb Q\Gamma'$, say $\delta=\sum_i\lambda_if_i$.

Define for $a\in\mathbb Q\Gamma$ the translation $t_a\in\Aut(\mathbb Q\Gamma)$ via $t_a(y):=y-(a,y)\delta$ and set $L:=\{t_a:a\in\mathbb Q\Gamma\}$. Note that $vt_a=t_{v(a)}v$ for all $v\in W$.

As in \cite[Chapter~6]{Ka1990} we can find $e_0\in E$ such that, writing
\[ E^0 := E\setminus\{e_0\}, \quad \Gamma^0 := \mathbb Z E^0, \quad
W^0 := W(\Gamma^0,E^0) \quad\textrm{and}\quad L^0 := \{t_a:a\in\mathbb Q\Gamma^0\}, \]
then $(\Gamma^0,E^0)$ is a generalised Cartan lattice of Dynkin type and $W$ is a subgroup of $W^0\ltimes L^0$ \cite[Proposition~6.5]{Ka1990}. Accordingly, we can write $w=\bar wt_a$. Since $w(y)-y\in\mathbb Z\delta$ for all $y\in\Gamma$, the same holds for $\bar w$. On the other hand, $\bar w(y)-y\in\Gamma^0$ for all $y\in\Gamma$. Since $\Gamma^0\cap\mathbb Z\delta=\{0\}$ we deduce that $\bar w=\id$ and $w=t_a$.

Observe now that there exist $a_i\in\mathbb Q\Gamma'$ such that
\[ w(y) = y - \sum_i(a_i,y)f_i \quad\textrm{for all }y\in\Gamma. \]
For, this clearly holds for each reflection $s_{f_i}$, so holds for all $v\in W'$ by induction on length.

Comparing this to the formula for $t_a$ and using that the $f_i$ are linearly independent gives us that $a_i-\lambda_ia\in\rad(\mathbb Q\Gamma)=\mathbb Q\delta$ for all $i$, and hence that $a\in\mathbb Q\Gamma'$. To complete the proof note that $w(f)=f$ for all $f\in\Gamma'$ implies $(a,f)=0$ for all $f\in\Gamma'$, and hence that $(a,a)=0$. On the other hand $a\in\mathbb Q\Gamma^0$, where the bilinear form is positive definite, so $a=0$. Hence $w=\id$.
\end{proof}

\begin{prop}\label{pr:weyl}
Let $(\Ga, E)$ be a generalised Cartan lattice and $F=(f_1,\ldots,f_r)$ a real exceptional sequence. Set $\Ga'=\bbZ F$ and choose $\sigma$ such that $\sigma F$ is orthogonal.  Then the assignment $w\mapsto w|_{\Ga'}$ induces an isomorphism
\begin{equation}\label{eq:weyl}
W(\Ga,E)\supseteq\langle s_{f_1},\ldots,s_{f_r}\rangle\stackrel{\sim}\lto W(\Ga',\sigma F).
\end{equation}
\end{prop}

\begin{proof} 
Let $w\in\langle s_{f_1},\ldots,s_{f_r}\rangle$ fix $\Ga'$ pointwise. We need to show that $w=\id$.

Extend $F$ to a complete real exceptional sequence $(f_1,\ldots,f_n)$ for $(\Gamma,E)$. Given $r<i\le n$ set $F_i=(f_1,\ldots,f_r,f_i)$ and $\Ga_i=\bbZ F_i$. By Lemma~\ref{le:orbits}~(1) there exists $\t_i\in\{\pm 1\}\wr B_{r+1}$ such that $(\Ga_i,\t_iF_i)$ is a generalised Cartan lattice, so we can apply Lemma~\ref{le:weyl} to deduce that $w$ is the identity on $\Ga_i$. This holds for each such $i$, so $w$ is the identity on the whole of $\Ga$.
\end{proof}

\begin{thm}\label{th:natural}
Let $\p\colon(\Ga',E')\to(\Ga,E)$ be a morphism of generalised Cartan lattices, so yielding an injection on the set of real roots. Then the map $s_a\mapsto s_{\p(a)}$ yields an injective group homomorphism $\p_*\colon W(\Ga',E')\to W(\Ga,E)$, and this restricts to the poset homomorphism $\NC(\Ga',E')\to\NC(\Ga,E)$ constructed previously. In particular, the map $\p_*$ identifies $\NC(\Ga',E')$ with $\{w\in\NC(\Ga,E)\mid w\le \p_*(\cox(\Ga',E'))\}$. 
\end{thm}

\begin{proof}
The map $\p_*$ is just the inverse of the isomorphism \eqref{eq:weyl} constructed above. For the second statement we just need to observe that if $F'$ is any real exceptional sequence in $(\Ga',E')$, then $\p_*(s_{F'})=s_{\p(F')}$.
\end{proof}

\subsection*{Pointed Coxeter groups}

For a Coxeter group $W=W(\Ga,E)$ we give a group theoretic description of the subgroups of $W$ which are of the form $W(\Ga',E')$ for some subobject $(\Ga',E')\subseteq(\Ga,E)$. If $W$ is finite, then these subgroups are known to be parabolic \cite[Lemma 1.4.3]{Be2003}. This is no longer true when $W$ is infinite. Moreover, the Coxeter diagram of such a subgroup is not necessarily obtained by removing vertices from the diagram of $W$. The authors are grateful to Christian Stump for suggesting the following example.

\begin{exm}\label{ex:parabolic}
Consider a Coxeter group $W=\langle s,t,u\rangle$  of affine type $\widetilde A_2$ with Coxeter element $c=stu$. The factorisation $c=s(tut)t$ yields a non-crossing partition $s(tut)$ and the corresponding subgroup $W'=\langle s,tut\rangle$ is affine of type $\widetilde A_1$. Thus $W'$ is not a parabolic subgroup.
\end{exm}

Let us define a \emph{pointed Coxeter group} as a triple $(W,S,c)$ consisting of a Coxeter system $(W,S)$ and a Coxeter element $c$. A pointed Coxeter group $(W',S',c')$ is a \emph{subgroup} of $(W,S,c)$ if $W'$ is a subgroup of $W$ and $\NC(W',c')=\{w\in \NC(W,c)\mid w\le c'\}$. Note that any such subgroup is determined by its Coxeter element since it is generated by its non-crossing partitions.

The following is an immediate consequence of Theorem~\ref{th:natural}.

\begin{cor}\label{co:subgroup}
Let $(\Ga,E)$ be a generalised Cartan lattice. Then $(W(\Ga,E),\cox(\Ga))$ is a pointed Coxeter group and the map $(\Ga',E')\mapsto (W(\Ga',E'),\cox(\Ga'))$ induces an order preserving bijection between the subobjects of $(\Ga,E)$ and the subgroups of $(W(\Ga,E),\cox(\Ga))$. \qed
\end{cor}

\begin{rem}
It is known that, in a Coxeter group, every subgroup generated by reflections is itself a Coxeter group \cite{Deo1989,Dy1990}. Thus if $w\in\NC(W,c)$ and $w=t_1\cdots t_r$ is a reduced expression as a product of reflections, then the subgroup generated by the $t_i$ is again a Coxeter group. However, it is not clear from their work that the subgroup depends only on $w$ and not on the choice of reduced expression, nor that $w$ is then a Coxeter element for the subgroup.
\end{rem}

\section{Non-crossing partitions revisited}

In this section we relate various properties of non-crossing partitions to our categorification.

\subsection*{The Kreweras complement}

Let $(\Ga, E)$  be a generalised Cartan lattice and
\[ \NC(\Ga,E) = \{w\in W(\Ga)\mid \id\le w\le \cox(\Ga)\} \]
the corresponding poset of non-crossing partitions. Given a sublattice $\De\subseteq\Ga$, we can form its orthogonal complements $^\perp\De$ and $\De^\perp$ with respect to $\langle -,-\rangle$. These induce two operations on $\Sub(\Ga,E)$ which correspond to the formation of Kreweras complements in $\NC(\Ga,E)$ via the isomorphism
\[ \cox\colon\Sub(\Ga,E)\stackrel{\sim}\lto\NC(\Ga,E). \]

\begin{prop}\label{pr:kreweras}
Taking $\De\subseteq\Ga$ to $^\perp\De$ and $\De^\perp$ induces two order reversing bijections $\Sub(\Ga,E)\to\Sub(\Ga,E)$ which are mutually inverse. Moreover,
\[ \cox(^\perp\De)=\cox(\De)^{-1}\cox(\Ga) \quad\text{and}\quad \cox(\De^\perp)=\cox(\Ga) \cox(\De)^{-1}. \]
\end{prop}

\begin{proof}
Let $(\De,E')$ be in $\Sub(\Ga,E)$. Thus $E'=(e_1,\ldots,e_r)$ is a subsequence of $\s E=(e_1,\ldots,e_n)$ for some $\s\in \{\pm1\}\wr B_n$. Set $E''=(e_{r+1},\ldots,e_n)$.  Then $^\perp\De=\bbZ E''$ and $(^\perp\De,E'')$ is in $\Sub(\Ga,E)$. We have $\cox(\Ga)=s_E=s_{E'}s_{E''}$ and therefore
\[ \cox(^\perp\De)=s_{E''}=s_{E'}^{-1}s_E=\cox(\De)^{-1}\cox(\Ga). \]
The argument for $\De^\perp$ is similar.
\end{proof}

\begin{cor}
Let $(\Ga,E)$ be a generalised Cartan lattice. Then every interval in $\NC(\Ga,E)$ is isomorphic to $\NC(\Ga',E')$ for some subobject $(\Ga',E')$ of $(\Ga,E)$.
\end{cor}

\begin{proof}
Let $[x,y]$ be an interval in $\NC(\Ga,E)$  and set $(\Ga',E')=\cox^{-1}(x^{-1}y)$. Then it follows from Proposition~\ref{pr:kreweras} that $[x,y]$ and $\NC(\Ga',E')$ are isomorphic.
\end{proof}

\subsection*{The Auslander-Reiten translate as a cyclic automorphism}

Let $(\Ga,E)$ be a generalised Cartan lattice with Coxeter element $c:=\cox(\Ga)$. Clearly conjugation by $c$ induces a poset automorphism on the set of non-crossing partitions $\NC(\Ga,E)$, having order the Coxeter number $h$ in Dynkin type. This automorphism has been much studied recently, especially in the context of \emph{cyclic sieving}, see for example \cite[Section~3.4.6]{Ar2009} as well as \cite{BR2011} and the references therein.

Now suppose that we have realised $(\Ga,E)$ as the Grothendieck group of a finite dimensional hereditary algebra $A$. Then we can define  the Auslander-Reiten translate $\tau=D\Ext^1_A(-,A)$ on the module category $\mod A$, see for example \cite[Section~4.12]{Be1991}, or more naturally as an exact autoequivalence of the derived category $\D^b(\mod A)$. In this case we have \[\Hom_{\D^b(\mod A)}(Y,\tau X)\cong D\Hom_{\D^b(\mod A)}(X,Y[1]),\] see \cite[Proposition 3.8]{Ha1987}, and hence that $\langle Y,\tau X\rangle =-\langle X,Y\rangle$ for all $X$ and $Y$.  Thus $[\tau X]=c[X]$, so the action of $\tau$ on exceptional sequences corresponds, via the isomorphism $\Exc(\Ga,E)\cong\NC(\Ga,E)$, to conjugation by the Coxeter element $c$.

\subsection*{Constructing non-crossing partitions}

It seems to be a hard problem to determine when an element of the Weyl group is a non-crossing partition, or equivalently to distinguish the real exceptional sequences inside the set of all exceptional sequences. One simplification is that it is enough to check this pairwise.

\begin{lem}\label{le:pairwise}
Let $(W,S)$ be a Coxeter system coming from a symmetrisable Cartan matrix and $c$ a Coxeter element. For a product $w=t_1\cdots t_r$ of reflections, the following are equivalent:
\begin{enumerate}
\item $w\in\NC(W,c)$ and $\ell(w)=r$.
\item $t_i\neq t_j$ and $t_it_j\in\NC(W,c)$ for all $i<j$.
\end{enumerate}
\end{lem}

\begin{proof}
This follows from Proposition~\ref{pr:exceptional-ncp}. For, if $A$ is a finite dimensional hereditary algebra, then a sequence $X=(X_1,\ldots,X_r)$ in $\mod A$ is exceptional if and only if $(X_i,X_j)$ is an exceptional pair for all $i<j$.
\end{proof}

Using our categorification, it is possible to construct an algorithm for determining whether or not a given sequence is a real exceptional sequence. This is based on Proposition~\ref{pr:Schofield}, together with the Derksen--Weyman algorithm \cite{DW2002}.

Let $(\Ga,E)$ be a generalised Cartan lattice, say with $E=(e_1,\ldots,e_n)$, and $a=(a_1,\ldots,a_r)$ an exceptional sequence of positive pseudo-real roots.
\begin{enumerate}
\item We begin by applying the Derksen--Weyman algorithm to $a_r$. This computes the \emph{canonical decomposition} of $a_r$, which is a sum of real and imaginary Schur roots. So, if the algorithm returns $a_r$, then we know that $a_r$ is a real Schur root.
\item We next compute the \emph{projective roots} $p_i=s_n\cdots s_{i+1}(e_i)$, as used in the proof of Theorem~\ref{th:Dynkin-type}, and apply the Derksen--Weyman algorithm to each of the roots $l_{a_r}(p_i)$ in turn. Let $M$ be indecomposable with $\dimv M=a_r$ and assume that $M$ is not projective. Then $l_{a_r}(p_i)$ is the dimension vector of the universal extension of $M$ by $P_i$, the projective module of dimension vector $p_i$. Thus the canonical decompositions of the $l_{a_r}(p_i)$ yield the dimension vectors of the indecomposable summands of the Bongartz complement $B$ of $M$, and $\Hom_A(M,B)=0$. If, on the other hand, $M$ is projective, then the Bongartz complement is just the sum of the other indecomposable projectives, and the $l_{a_r}(p_i)$ are the dimension vectors of the cokernels of the minimal right $\mathrm{add}(M)$-approximations of the projectives. It therefore follows as in \cite[Theorem~16]{Hu2011} that, in both cases, the Derksen-Weyman algorithm applied to the $l_{a_r}(p_i)$ produces $n-1$ real Schur roots, say $\bar p_1,\ldots,\bar p_{n-1}$ other than $a_r$, and that these are the projective roots for $a_r^\perp$.
\item From the $\bar p_i$, we can easily construct an orthogonal exceptional sequence $\bar E:=(\bar e_1,\ldots,\bar e_{n-1})$ such that $(\bar E,a_r^\perp)$ is a subobject of $(\Ga,E)$. We now express each of $a_1,\ldots,a_{r-1}$ in terms of $\bar E$, noting that each $a_i$ must be a positive linear combination of the $\bar e_i$, otherwise $a$ will not be a real exceptional sequence by Corollary~\ref{co:lies-in-C(X)}. Now repeat these steps using the sequence $\bar E$ inside the sublattice $a_r^\perp$.
\end{enumerate}

\begin{exm}\label{exm:Schofield}
Following Schofield \cite{Sc1992}, consider the path algebra of the quiver
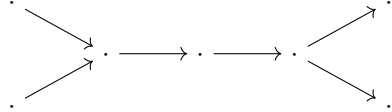
\[ \begin{tikzcd}[row sep=tiny]
\mathbf{\cdot} \arrow{dr} &&&& \mathbf{\cdot}\\
& \mathbf{\cdot} \arrow{r} & \mathbf{\cdot} \arrow{r} & \mathbf{\cdot} \arrow{ur} \arrow{dr}\\
\mathbf{\cdot} \arrow{ur} &&&& \mathbf{\cdot}
\end{tikzcd}\]
and take exceptional modules $X$ and $Y$ having dimension vectors
\[ x = \dimvector1121211 \quad\textrm{and}\quad y =\dimvector0011100 \]
Then $X$ and $Y$ both lie in the inhomogeneous tube of length four, have quasi-length three, and $c^2(x)=y$. It follows that $\dim\Hom(X,Y)=\dim\Ext^1(X,Y)=1$, and similarly $\dim\Hom(Y,X)=\dim\Ext^1(Y,X)=1$, so both $\langle x,y\rangle=0=\langle y,x\rangle$ but neither $(X,Y)$ nor $(Y,X)$ is an exceptional pair.

In terms of our algorithm, we have
\begin{center}
\begin{tabular}{cccccccc}
\toprule
$i$ & 1 & 2 & 3 & 4 & 5 & 6 & 7\\
\midrule
$p_i$ & $\dimvector1011111$ & $\dimvector0111111$ & $\dimvector0011111$ & $\dimvector0001111$
& $\dimvector0000111$ & $\dimvector 0000010$ & $\dimvector0000001$\\
\midrule
$l_x(p_i)$ & $\dimvector2132322$ & $\dimvector1232322$ & $\dimvector2253533$ & $\dimvector2243533$
& $\dimvector1121322$ & $\dimvector1121221$ & $\dimvector1121212$\\
\bottomrule
\end{tabular}
\end{center}
Thus the canonical decompositions of the $l_x(p_i)$ are
\[ l_x(p_1) = \bar p_1 + \bar p_5 \qquad l_x(p_2) = \bar p_2 + \bar p_5 \qquad l_x(p_3) = \bar p_1 + \bar p_2 + \bar p_5 + \bar p_6 \]
\[ l_x(p_4) = \bar p_1 + \bar p_2 + \bar p_5 \qquad l_x(p_5) = \bar p_3 + \bar p_4 + \bar p_5 \qquad l_x(p_6) = \bar p_3 + \bar p_5 \qquad l_x(p_7) = \bar p_4 + \bar p_5 \]
where
\begin{center}
\begin{tabular}{ccccccc}
\toprule
$i$ & 1 & 2 & 3 & 4 & 5 & 6\\
\midrule
$\bar p_i$ & $\dimvector1011211$ & $\dimvector0111211$ & $\dimvector0000110$ & $\dimvector0000101$
& $\dimvector1121111$ & $\dimvector0010000$\\
\midrule
$\bar e_i$ & $\dimvector1011000$ & $\dimvector0111000$ & $\dimvector0000110$ & $\dimvector0000101$
& $\dimvector1111111$ & $\dimvector0010000$\\
\bottomrule
\end{tabular}
\end{center}
In particular $y=\bar e_1+\bar e_2+\bar e_3+\bar e_4-\bar e_5$, so is not a positive linear combination.

Note also that one cannot just take the minimal positive elements in $x^\perp$ as the simples. For, such a collection must contain $\dimvector0001100$ instead of $\bar e_5$.
\end{exm}

\section{Hereditary categories}

We consider the category $\mathfrak H$ of hereditary abelian categories arising in the represention theory of algebras. More precisely, the objects in $\mathfrak H$ are the categories $\mod A$ of finitely generated modules over an hereditary artin algebra $A$.\footnote{
The centre of an hereditary artin algebra $A$ is semisimple, and $A$ is said to be \emph{connected} if the centre is a field, say $k$. In that case $A$ is actually a finite dimensional $k$-algebra.
}
The morphisms in $\mathfrak H$ are fully faithful exact functors, modulo natural isomorphisms, having an extension closed essential image.

\subsection*{Homological epimorphisms}

Following \cite{GL1991}, a ring homomorphism $A\to B$ is called a \emph{homological epimorphism} if restriction of scalars induces an
isomorphism \[\Ext_{B}^n(X,Y)\xto{\sim}\Ext_{A}^n(X,Y)\] for all $n\ge 0$ and all $B$-modules $X,Y$. The homomorphism is \emph{finite} if $B$ is finitely generated when viewed as an $A$-module.

\begin{lem}\label{le:epi}
Let $f\colon A\to B$ be a finite homological epimorphism between artin algebras. Then the restriction of scalars $\mod B\to\mod A$ is a full exact embedding having a left adjoint, whose essential image is closed under kernels, cokernels and extensions.

Conversely, if $A$ is hererditary, then every such subcategory of $\mod A$ arises in this way.
\end{lem}

\begin{proof}
Since $f$ is finite, the restriction of scalars functor goes between the categories of finitely generated modules. Moreover, this functor is always exact and has extension of scalars as a left adjoint. Finally, since $f$ is a homological epimorphism, the functor is fully faithful and the essential image is closed under kernels, cokernels and extensions.

Conversely, let $A$ be hereditary and $\C\subseteq\mod A$ a full subcategory closed under kernels, cokernels and extensions. Then $\Ext^n_\C(X,Y)\cong\Ext^n_A(X,Y)$ for all $n\geq0$ and all $X,Y\in\C$. If $\C$ is also a reflective subcategory, so the inclusion has a left adjoint $L$, then $LA$ is a projective generator for $\C$. Thus $\C$ is equivalent to $\mod\End_A(LA)$, and the map $\End_A(A)\to\End_A(LA)$ induced by the functor $L$ yields a homological epimorphism $f\colon A\to\End_A(LA)$. Finally, the isomorphism $\Hom_A(A,LA)\cong\End_A(LA)$ is now an isomorphism of right $A$-modules, showing that $f$ is finite.

In fact, using the isomorphisms $LA\cong\Hom_A(A,LA)\cong\End_A(LA)$, we can endow $LA$ with the structure of an algebra, in which case the algebra homomorphism $f\colon A\to LA$ corresponds to the identity in $\End_A(LA)$.
\end{proof}

Recall from Lemma~\ref{le:hereditary} that every generalised Cartan lattice can be realised as $K_0(A)$ for some hereditary artin algebra $A$.

\begin{thm}\label{th:epi}
If $f\colon A\to A'$ is a finite homological epimorphism between hereditary artin algebras, then restriction of scalars induces a morphism $f^*\colon K_0(A')\to K_0(A)$ of generalised Cartan lattices.

Conversely, every morphism of generalised Cartan lattices ending in $K_0(A)$ is up to isomorphism of the form $f^*\colon K_0(A')\to K_0(A)$ for some finite homological epimorphism $f\colon A\to A'$ between hereditary artin algebras.
\end{thm}

\begin{proof}
Let $f\colon A\to A'$ be a homological epimorphism. Then $f^*$ is necessarily an isometry. Moreover, restriction of scalars via $f$ sends exceptional sequences in $\mod A'$ to exceptional sequences in $\mod A$, so by Propositions~\ref{pr:morph} and \ref{pr:exceptional-ncp} the map $f^*$ is a morphism of generalised Cartan lattices.

Conversely, by Theorem~\ref{th:main} together with Proposition~\ref{pr:exceptional-ncp}, every morphism of generalised Cartan lattices ending in $K_0(A)$ is equivalent to one of the form $\p\colon K_0(\C(X))\hookrightarrow K_0(A)$ where $X$ is an exceptional sequence in $\mod A$. By Theorem~\ref{th:subcat} there is a finite homological epimorphism $f\colon A\to A'$ such that restriction of scalars identifies $\mod A'$ with $\C(X)$, and hence $\p=f^*$.
\end{proof}

\begin{cor}\label{co:faithful}
Let $\mathfrak C$ denote the category of generalised Cartan lattices. Taking an hereditary abelian category $\C$ to its Grothendieck group $K_0(\C)$ induces a faithful functor $\mathfrak H\to\mathfrak C$ which reflects isomorphisms.
\end{cor}

\begin{proof}
Fix a functor $\p\colon\mod A'\to\mod A$ representing a morphism in $\mathfrak H$, yielding the morphism of generalised Cartan lattices $\p_*\colon K_0(A')\to K_0(A)$, $[X]\mapsto[\p(X)]$. Also, $\p$ is equivalence if and only if $\p_*$ is an isomorphism. Let $\psi\colon\mod A'\to\mod A$ be another such functor, and suppose that $\p_*=\psi_*$. Then $[\p(A')]=[\psi(A')]$, so $\p(A')\cong \psi(A')$ by Lemma~\ref{le:rigid-dim-vector}. Since $\p$ is naturally isomorphic to $-\otimes_{A'} \p(A')$, and similarly for $\psi$, we conclude that $\p\cong\psi$.
\end{proof}

It is clear that the functor $\mathfrak H\to \mathfrak C$ is not full. For, fix two fields $k$ and $k'$. Then $K_0(k')\cong K_0(k)$ in $\mathfrak C$, but $\mod k'\cong\mod k$ in $\mathfrak H$ if and only if $k'\cong k$. Also, each derived equivalence $\D^b(\mod A')\xto{\sim} \D^b(\mod A)$ induces an isomorphism $K_0(A')\xto{\sim}K_0(A)$ in $\mathfrak C$, but again $\mod A'$ and $\mod A$ need not to be equivalent in $\mathfrak H$. We do however have the following result.

\begin{thm}[Happel {\cite[Theorem~5.12]{Ha1987}}]\label{th:Happel}
Let $A$ be a finite dimensional algebra over an algebraically-closed field $k$. Assume that $A$ is basic and simply-connected. Then $A$ is derived equivalent to the path algebra $kQ$ of a Dynkin quiver $Q$ if and only if their generalised Cartan lattices are isometric $K_0(A)\cong K_0(kQ)$. \qed
\end{thm}

We can now express Theorem~\ref{th:main} in terms of a fixed hereditary artin algebra $A$. In this case set $\Sub(\mod A)$ to be the set of subcategories of $\mod A$ of the form $\C(X)$ for some exceptional sequence $X$. The poset structure on $\Sub(\mod A)$ is given by inclusion of subcategories.

\begin{cor}\label{co:epi}
Let $A$ be a hereditary artin algebra. Then there is a natural isomorphism of posets $\Sub(\mod A)\cong\NC(K_0(A))$. In particular, two exceptional sequences $X$ and $Y$ are equivalent under the braid group action if and only if $\cox(\C(X))=\cox(\C(Y))$. \qed
\end{cor}

\subsection*{Non-crossing partitions and Hom-free sets}

Let $A$ be an hereditary artin algebra. By a \emph{Hom-free set} in $\mod A$ we will mean a set $\{S_1,\ldots,S_r\}$ of exceptional modules satisfying $\Hom_A(S_i,S_j)=0$ for all $i\neq j$. Such Hom-free sets were studied by various authors \cite{BRT2012,GP1987,Re2007,Rie1980,Si2012}.

\begin{prop}\label{pr:hom-free}
Let $X$ be an exceptional sequence in $\mod A$. Then the map sending $\C(X)$ to the set of simple objects in $\C(X)$ induces an injective map from $\NC(K_0(A))$ to the Hom-free sets up to isomorphism. If $A$ is representation-finite, then this map is a bijection.
\end{prop}

\begin{proof}
If $X$ is an exceptional sequence, then $\C(X)\cong\mod B$ for some hereditary artin algebra $B$, by Theorem~\ref{th:subcat}, and
  hence the simple objects in $\C(X)$ form a Hom-free set. Conversely, $\C(X)$ is uniquely determined by its set of simple objects. Using Corollary~\ref{co:epi} we know that the non-crossing partitions are in bijection with the subcategories of the form $\C(X)$ for $X$ exceptional. Thus the map sending $\C(X)$ to its set of simple objects induces an injective map from the set of non-crossing partitions to the Hom-free sets up to isomorphism.

If now $A$ is representation-finite, then any Hom-free set can be arranged to form an exceptional sequence. For, the (isomorphism classes of) indecomposable $A$-modules are partially ordered using the Auslander-Reiten quiver, so $\Ext^1_A(M,N)\neq0$ implies $N\prec M$. It follows that every Hom-free set arises as the set of simple objects in $\C(X)$ for some exceptional sequence $X$.
\end{proof}

When $A$ is representation-finite, this observation can be used to enumerate the non-crossing partitions. This was already done by Gabriel and de la Pe\~na in \cite[Sections~2.6, 2.7]{GP1987} and from `painful calculations' they obtained the Coxeter-Catalan numbers of
$ADE$-type.\footnote{
Note however that their form for the Coxeter-Catalan number of type $D$ has not been simplified, and although they have the correct number for $E_8$, their numbers for $E_6$ and $E_7$ are slightly wrong.
}

\begin{exm}
The map from non-crossing partitions to Hom-free sets is in general not surjective. For, consider the path algebra of the quiver
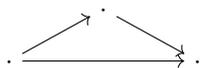
\[ \begin{tikzcd}[row sep=tiny]
& \mathbf{\cdot} \arrow{dr}\\
\mathbf{\cdot} \arrow{ur} \arrow{rr} && \mathbf{\cdot}
\end{tikzcd} \]
Then the exceptional modules $X$ and $Y$ having dimension vectors
\[ [X]=\begin{matrix}\;\;\;1\;\;\;\\[-0.5em]0\;\;\;\;\;\;0\end{matrix} \qquad\textrm{and}\qquad 
[Y]=\begin{matrix}\;\;\;0\;\;\;\\[-0.5em]1\;\;\;\;\;\;1\end{matrix} \]
form a Hom-free set, but neither $(X,Y)$ nor $(Y,X)$ is an exceptional pair.
\end{exm}

Let $F=(f_1,\ldots,f_r)$ be a real exceptional sequence in a generalised Cartan lattice $(\Ga,E)$. Writing  $f_i=\sum_{j}\lambda_{ij}e_j$ in terms of $E$ we may define the \emph{height} of $F$ to be $\height(F):=\sum_{ij}|\lambda_{ij}|$.

\begin{lem}
Let $(\Ga,E)$ be a generalised Cartan lattice, and $F$ a real exceptional sequence in $\Ga$. If $F'=\sigma F$ has minimal height in the orbit of $F$, then the roots in $F'$ are unique up to sign.
\end{lem}

\begin{proof}
By Lemma~\ref{le:hereditary} we can realise $(\Ga,E)$ as $K_0(A)$ for some hereditary artin algebra $A$, and by Proposition~\ref{pr:exceptional-ncp} we can lift $F$ (up to sign) to an exceptional sequence $X=(X_1,\ldots,X_r)$ in $\mod A$. Then $\height(F)$ equals the length $\ell_A(X):=\ell_A(X_1\oplus\cdots\oplus X_r)$. Thus if $F'$ has minimal height, then the roots in $F'$ are precisely the classes of the simple objects in $\C(X)$.
\end{proof}

\begin{appendix}
\section{Perpendicular calculus}

In this appendix we collect some basic facts about perpendicular categories, following Geigle--Lenzing \cite[Section~3]{GL1991}, Schofield \cite[Section~2]{Sc1991}, and Crawley-Boevey \cite{CB1992}. 

Let $A$ be an hereditary artin algebra. We consider the category $\mod A$ of finitely generated $A$-modules.

For a subset $X\subseteq\mod A$ we define the right and left \emph{perpendicular categories} to be the full subcategories
\begin{align*}
X^\perp &:= \{ Y \in\mod A \mid \Hom_A(X_i,Y)=0=\Ext^1_A(X_i,Y) \textrm{ for all }X_i\in X \}\\
{}^\perp X &:= \{ Y \in \mod A \mid \Hom_A(Y,X_i)=0=\Ext^1_A(Y,X_i) \textrm{ for all }X_i\in X \}
\end{align*}
These are clearly closed under kernels, cokernels and extensions. We also consider $\C(X)$, the smallest full subcategory of $\mod A$ closed under kernels, cokernels and extensions and containing each $X_i\in X$.

The following proposition describes simultaneously  adjoints of the inclusions $\C(X)\to\mod A$ and $\C(X)^\perp \to\mod A$ via a five term exact sequence; see also \cite[Proposition~3.5]{GL1991} and \cite[Theorem~2.2]{KS2010}.

\begin{prop}\label{pr:perp1}
Let $\C\subseteq\mod A$ be a full subcategory closed under kernels, cokernels and extensions. Assume that $X\in\C$ is a relative projective generator.
\begin{enumerate}
\item Each $A$-module $M$ fits into a functorial exact sequence
\begin{equation}\label{eq:perp}
0\lto \bar M^{\C^\perp} \lto M_\C \lto M\lto M^{\C^\perp} \lto \bar M_\C \lto 0
\end{equation}
with $M_\C,\bar M_\C\in\C$ and $M^{\C^\perp},\bar M^{\C^\perp}\in \C^\perp$.
\item The map $M\mapsto M^{\C^\perp}$ yields a left adjoint for the inclusion $\C^\perp\to\mod A$.
\item The map $M\mapsto M_\C$ yields a right adjoint for the inclusion $\C\to\mod A$.
\item $\C={}^\perp(\C^\perp)$.
\end{enumerate}
\end{prop}

\begin{proof}
We first note that $\C=\C(X)$ and $\C^\perp=X^\perp$.

(1) Given $M\in\mod A$ we set $X_M\to M$ to be a right $\add(X)$-approximation of $M$, so $X_M\in\add(X)$ and $\Hom_A(X',X_M)\to\Hom_A(X',M)$ is surjective for all $X'\in\add(X)$. Note that its kernel $L$ satisfies $\Ext^1_A(X',L)=0$ and its cokernel $N$ satisfies $\Hom_A(X',N)=0$ and $\Ext^1_A(X',N)\cong\Ext^1_A(X',M)$ for all $X'\in\add(X)$.

Consider the push-out diagram
\[ \begin{CD}
@. X_L @= X_L\\
@. @VVV @VVV\\
0 @>>> L @>>> X_M @>>> M\\
@. @VVV @VVV @|\\
0 @>>> \bar M^{\C^\perp} @>>> M_\C @>>> M\\
@. @VVV @VVV\\
@. 0 @. 0
\end{CD} \]
It follows from the comments above that $M_\C\in\C$ and $\bar M^{\C^\perp}\in\C^\perp$.

Next let
\[ \e \colon 0 \lto N \lto E \lto X_N^1 \lto 0 \]
be a universal extension, so $X_N^1\in\add(X)$ and $\Hom_A(X',X_N^1)\to\Ext_A^1(X',N)$ is surjective for all $X'\in\add(X)$. It follows that $\Ext^1_A(X',E)=0$ for all $X'\in\add(X)$.

Consider a right $\add(X)$-approximation $X_E\to E$ and let $F$ be the image of the composition $X_E\to E\to X_N^1$. Then $F$ is in $\C$, and is relative projective since it is a submodule of $X_N^1$, so $F\in\add(X)$. The Snake Lemma now yields an exact commutative diagram
\[ \begin{CD}
0 @>>> G @>>> X_E @>>> F @>>> 0\\
@. @VVV @VVV @VVV\\
0 @>>> N @>>> E @>>> X_N^1 @>>> 0\\
@. @VVV @VVV @VVV\\
0 @>>> N' @>>> M^{\C^\perp} @>>> \bar M_\C @>>> 0\\
@. @VVV @VVV @VVV\\
@. 0 @. 0 @. 0
\end{CD} \]
Again, $\bar M_\C\in\C$ and $M^{\C^\perp}\in\C^\perp$. Also, since the top row is split and $\Hom_A(X',N)=0$ for all $X'\in\add(X)$, the map $G\to N$ is zero, so $N'\cong N$.

Putting this together yields the five term exact sequence
\[ 0 \lto \bar M^{\C^\perp} \lto M_\C \lto M \lto M^{\C^\perp} \lto \bar M_\C \lto 0. \]

(2) Since $M^{\C^\perp},\bar M^{\C^\perp}\in\C^\perp$ we deduce that $\Hom_A(Y,M_\C)\cong\Hom_A(Y,M)$ for all $Y\in\C$.

(3) Analogously, $\Hom_A(M^{\C^\perp},Y)\cong\Hom_A(M,Y)$ for all $Y\in\C^\perp$.

(4) It is clear that $\C\subseteq{}^\perp(\C^\perp)$, so let $M\in{}^\perp(\C^\perp)$ and consider the five term exact sequence. The morphism $M\to M^{\C^\perp}$ is zero, so we are left with an extension
\[ 0 \lto \bar M^{\C^\perp} \lto M_\C \to M \lto 0. \]
This must split, so $\bar M^{\C^\perp}=0$ and $M\cong M_\C$ lies in $\C$.
\end{proof}

\begin{rem}\label{re:perp}
We observe that if $X\in\mod A$ is an exceptional module, then $\C=\C(X)=\add(X)$, so has $X$ itself as a relative projective generator. In this case, given $M$, we can take a minimal right approximation
\[ X_M = \Hom_A(X,M)\otimes_{\End_A(X)}X \lto M, \]
and the kernel already lies in $X^\perp$. Similarly we can take a minimal universal extension using $X_M^1=\Ext_A^1(X,M)\otimes_{\End_A(X)}X$, and the push-out to $M'$ has middle term in $X^\perp$. It follows that
\[ M_\C = \Hom_A(X,M)\otimes_{\End_A(X)}X \quad\textrm{and}\quad \bar M_\C = \Ext_A^1(X,M)\otimes_{\End_A(X)} X. \]
\end{rem}

There is an analogue of Proposition~\ref{pr:perp1} describing the right adjoint of the inclusion $^\perp \C\to\mod A$; we state this for completeness.

\begin{prop}\label{pr:perp2}
Let $\C\subseteq\mod A$ be a full subcategory closed under kernels, cokernels and extensions. Assume that $X\in\C$ is a relative injective cogenerator.
\begin{enumerate}
\item Each $A$-module $M$ fits into a functorial exact sequence
\begin{equation}\label{eq:perp2}
0\lto \bar M^\C \lto M_{{}^\perp\C} \lto M \lto M^\C \lto \bar M_{{}^\perp\C} \lto 0 
\end{equation}
with $M^\C,\bar M^\C\in\C$ and $M_{^\perp\C},\bar M_{^\perp\C}\in {}^\perp\C$. 
\item The map $M\mapsto M_{^\perp\C}$ yields a right adjoint for the inclusion ${^\perp\C}\to\mod A$.
\item The map $M\mapsto  M^\C$ yields a left adjoint for the inclusion $\C\to\mod A$.
\item $\C=({}^\perp\C)^\perp$. \qed
\end{enumerate} 
\end{prop}

\begin{thm}\label{th:subcat}
Let $A$ be an hereditary artin algebra and $\C\subseteq\mod A$ a full subcategory. Then the following are equivalent.
\begin{enumerate}
\item The inclusion $\C\to\mod A$ admits a left adjoint and $\C$ is closed under kernels, cokernels, and extensions.
\item[($1^\prime$)] The inclusion $\C\to\mod A$ admits a right adjoint and $\C$ is closed under kernels, cokernels, and extensions.
\item There is a finite homological epimorphism $A\to B$ such that restriction of scalars induces an equivalence $\mod B\xto{\sim}\C$.
\item There is an exceptional sequence $X$ in $\mod A$ such that $\C=\C(X)$.
\item There is an exceptional sequence $Y$ in $\mod A$ such that $\C=Y^\perp$.
\item[($4^\prime$)] There is an exceptional sequence $Z$ in $\mod A$ such that $\C={^\perp Z}$.
\end{enumerate}
In this case $X$ has length $\rk K_0(B)$, while $Y$ and $Z$ have length $\rk K_0(A)-\rk K_0(B)$.
\end{thm}

\begin{proof}
(1) $\Leftrightarrow$ (2): This follows from Lemma~\ref{le:epi}. Using the duality $(\mod A)^\op\xto{\sim}\mod A^\op$ we obtain ($1^\prime$) $\Leftrightarrow$ (2).

(2) $\Rightarrow$ (3): By Lemma~\ref{le:hereditary} we have a complete, orthogonal exceptional sequence $S=(S_1,\ldots,S_r)$ in $\mod B$ consisting of simple $B$-modules, and clearly $\mod B=\C(S)$. Setting $X_i\in\C$ to be the image of $S_i$, we deduce that $X=(X_1,\ldots,X_r)$ is an exceptional sequence and $\C(X)=\C$. Moreover note that $\mod B$, and hence $\C$, has both a relative projective generator and a relative injective cogenerator.

(4) $\Rightarrow$ (1): Let $Y=(Y_1,\ldots,Y_r)$ be an exceptional sequence, and set $\bar Y:=(Y_1,\ldots,Y_{r-1})$. Note that we can view $\bar Y$ as an exceptional sequence in $Y_r^\perp$, in which case its relative right perpendicular category is $\bar Y^\perp\cap Y_r^\perp=Y^\perp$. Now, by Remark~\ref{re:perp} we know that $\C(Y_r)=\add(Y_r)$, so by Proposition~\ref{pr:perp1} the inclusion $Y_r^\perp\to\mod A$ has a left adjoint, and hence by (1) $\Rightarrow$ (2) we have $Y_r^\perp\cong\mod B$ for some finite homological epimorphism $A\to B$. Thus by induction on the length of $Y$ we know that $Y^\perp\to Y_r^\perp$ also has a left adjoint. Composing these yields a left adjoint to the inclusion $Y^\perp\to\mod A$. The implication ($4'$) $\Rightarrow$ ($1'$) is dual.

(3) $\Rightarrow$ (4): Let $X=(X_1,\ldots,X_r)$ be an exceptional sequence. We claim that $\C(X)={}^\perp(X^\perp)=({}^\perp X)^\perp$. We prove the first identity, the second being dual.

As above, set $\bar X:=(X_1,\ldots,X_{r-1})$, an exceptional sequence in $X_r^\perp$. The relative right perpendicular category of $\bar X$ in $X_r^\perp$ is $\bar X^\perp\cap X_r^\perp=X^\perp$, and the relative left perpendicular category of the latter is ${}^\perp(X^\perp)\cap X_r^\perp$. On the other hand, $X_r^\perp\cong\mod B$ for some finite homological epimorphism $A\to B$, so by induction on the length of $X$ we know that ${}^\perp(X^\perp)\cap X_r^\perp=\C(\bar X)$.

Now apply Proposition~\ref{pr:perp1} to a module $M\in{}^\perp(X^\perp)$ with respect to the category $\C(X_r)$. This yields a five term exact sequence
\[ 0 \to M_0 \to M_1 \to M \to M_2 \to M_3 \to 0 \]
with $M_1,M_3\in\C(X_r)\subset{}^\perp(X^\perp)$, and hence $M_0,M_2\in X_r^\perp\cap{}^\perp(X^\perp)=\C(\bar X)$. Thus $M_i\in\C(X)$ for all $i$, whence $M\in\C(X)$. This proves the claim.

Now, to prove the implication (3) $\Rightarrow$ (4), set $\C':={}^\perp X$ and note that $\C(X)=({}^\perp X)^\perp=(\C')^\perp$. Using the implication ($4'$) $\Rightarrow$ (3) we know that $\C'=\C(Y)$ for some exceptional sequence $Y$, and hence $\C(X)=Y^\perp$ as required. The implication (3) $\Rightarrow$ ($4'$) is dual.
\end{proof}

\begin{rem}\label{re:subcat}
Let $\C\subseteq\mod A$ be a full subcategory closed under kernels, cokernels, and extensions. Then the inclusion $\C\to\mod A$ admits left and right adjoints if and only if the Loewy lengths of the objects in $\C$ are bounded \cite[8.2]{Ga1973}. This holds, for example, if $A$ is representation-finite or if $\C$ is finitely generated.
\end{rem}

We now list some useful consequences of this result. 

\begin{cor}
Let $X$ be an exceptional sequence in $\mod A$. Then $\C(X)$ contains both a relative projective generator and a relative injective cogenerator, so Propositions~\ref{pr:perp1} and \ref{pr:perp2} hold for $\C(X)$.
\end{cor}

\begin{proof}
We know that $\C(X)\cong\mod B$ for some hereditary artin algebra $B$.
\end{proof}

Recall that an exceptional sequence $X=(X_1,\ldots,X_r)$ in $\mod A$ is called complete provided that $r=\rk K_0(A)$.

\begin{cor}\label{co:complete}
Let $X$ be an exceptional sequence in $\mod A$. Then $\mod A=\C(X,X^\perp)$. In particular,
\[ X \textrm{ is complete} \quad \iff \quad \C(X) = \mod A \quad \iff \quad X^\perp = 0. \]
\end{cor}

\begin{proof}
Consider the five term exact sequence given in Proposition~\ref{pr:perp1} with respect to the subcategory $\C(X)$. It follows that $\mod A=\C(X,X^\perp)$ and also that $\C(X)=\mod A$ if and only if $X^\perp=0$. On the other hand, we know from Lemmas~\ref{le:lin-indept} and \ref{le:except} that the rank of $K_0(\C(X))$ equals the length of $X$ and that $K_0(\mod A)=K_0(\C(X))\oplus K_0(X^\perp)$. Therefore $X$ is complete if and only if $X^\perp=0$.
\end{proof}

\begin{cor}[{\cite[Lemma~1]{CB1992}}]\label{co:extend}
Let $(X_1,\ldots,X_r)$ be an exceptional sequence in $\mod A$ which is not complete. Then for each $0\leq i\leq r$ there exists an $A$-module $Y$ such that $(X_1,\ldots,X_i,Y,X_{i+1},\ldots,X_r)$ is exceptional. Moreover, if this is complete, then $Y$ is unique up to isomorphism.
\end{cor}

\begin{proof}
We have that ${}^\perp(X_1\oplus\cdots\oplus X_i)\cap(X_{i+1}\oplus\cdots\oplus X_r)^\perp$ is equivalent to $\mod B$ for some hereditary artin algebra $B$ of rank $\rk K_0(A)-r$, so let $Y$ be any exceptional $B$-module. If this is complete, then $\rk K_0(A)=r+1$, so $\rk K_0(B)=1$ and $Y$ is unique up to isomorphism.
\end{proof}

\begin{cor}\label{co:lies-in-C(X)}
Let $X=(X_1,\ldots,X_r)$ be an exceptional sequence in $\mod A$. Then a rigid module $Y$ lies in $\C(X)$ if and only if there exists some $Z\in\C(X)$ with $[Z]=[Y]$. In particular, if $X$ is orthogonal, then this happens if and only if $[Y]$ is a non-negative linear combination of the $[X_i]$.
\end{cor}

\begin{proof}
Consider the five term exact sequence
\[ 0 \to B \to C \to Y \to B' \to C' \to 0 \]
coming from Proposition~\ref{pr:perp1}, where $B,B'\in\C(X)^\perp$ and $C,C'\in\C(X)$. Let $I$ be the image of the map $Y\to B'$. Then $\Ext^1_A(Y,Y)\twoheadrightarrow\Ext^1_A(Y,I)$ implies that $\Ext^1_A(Y,I)=0$, and hence that $\langle[Y],[I]\rangle\geq0$. On the other hand, if $Z\in\C(X)$, then $\Hom_A(Z,I)\hookrightarrow\Hom_A(Z,B')=0$ implies that $\Hom_A(Z,I)=0$, and hence that $\langle[Z],[I]\rangle\leq0$.

Now, if $[Z]=[Y]$, then $\langle[Y],[I]\rangle=0$, and so $I=0$. Thus the five term sequence degenerates to a short exact sequence $0\to B\to C\to Y\to 0$, in which case $[B]=[C]-[Y]\in K_0(\C(X))\cap K_0(\C(X)^\perp)=0$, so that $B=0$ and $Y\cong C\in\C(X)$.
\end{proof}

The following result, due to Schofield \cite{Sc1990} (see also \cite{Ri1996}), can be viewed as a reduction theorem for constructing exceptional modules.

\begin{prop}[\cite{Sc1990}]\label{pr:Schofield}
Let $X$ be a non-simple exceptional module. Then there exists an othogonal exceptional pair $(U,V)$ and a non-split short exact sequence of the form
\[ 0 \to V^b \to X \to U^a \to 0. \]
\end{prop}

\begin{proof}
Given a proper submodule $M$ of $X$, consider the five term exact sequence from Proposition~\ref{pr:perp1}~(1) for $M$ with respect to the subcategory $\C={}^\perp X$. If the map $M_\C\to M$ is non-zero, then so too is the composition $M_\C\to M\to X$, a contradiction since $M_\C\in{}^\perp X$. Thus the sequence degenerates to give $0\to M\to X^r\to N\to 0$ for some non-zero $N\in\C$. It follows, by applying $\Hom_A(N,-)$, that $\Ext^1_A(N,N)=0$. Thus, taking any indecomposable summand $Y$ of $N$, we get that $(X,Y)$ is an exceptional pair such that $Y$ is generated by $X$ (that is, is a factor of some $X^r$). In particular, $X$ is not a simple object in $\C(X,Y)$. By Theorem~\ref{th:subcat}~(2) we can write $\C(X,Y)=\C(U,V)$ for some orthogonal exceptional pair $(U,V)$, and clearly every non-simple object $L\in\C(U,V)$ fits inside a non-split short exact sequence of the form $0\to V^b\to L\to U^a\to 0$.
\end{proof}

\section{Crystallographic Coxeter Groups}

Let $(W,S)$ be a Coxeter system. As usual, for $s,t\in S$ denote the order of $st$ by $m_{st}$. Following \cite[Sections~5,6]{Hum1990} we can define `the geometric representation' of $W$ by taking a real vector space $V$ with basis $e_s$ for $s\in S$ equipped with the symmetric bilinear form $(e_s,e_t):=-\cos\big(\pi/m_{st}\big)$, with the convention that this equals $-1$ whenever $m_{st}=\infty$. There is then a faithful representation $\sigma\colon W\to\mathrm{GL}(V)$ sending $s$ to the reflection $\sigma_s\colon\lambda\mapsto\lambda-2(e_s,\lambda)e_s$. We may then declare $(W,S)$ to be a \emph{crystallographic} Coxeter group (relative to $\sigma$) provided $W$ stabilises a lattice in $V$. This leads to the following result.

\begin{prop}[{\cite[Proposition~6.6]{Hum1990}}]
A Coxeter system $(W,S)$ is crystallographic if and only if
\begin{enumerate}
\item $m_{st}\in\{2,3,4,6,\infty\}$ for all $s\neq t$ in $S$, and
\item in each circuit of the Coxeter graph, the number of edges labelled 4 (resp. 6) is even. \qed
\end{enumerate}
\end{prop}

On the other hand, the term ``crystallographic Coxeter group'' is also used in the literature to describe those groups which arise as the Weyl group of a Kac--Moody Lie algebra. This occurs if and only if condition (1) above holds, so $m_{st}\in\{2,3,4,6,\infty\}$ for all $s\neq t$ in $S$ (use \cite[Proposition~3.13]{Ka1990}).

Here we are interested in Weyl groups of \emph{symmetrisable} Kac--Moody Lie algebras. It is therefore of interest to have an equivalent characterisation of these groups.

\begin{thm}
A Coxeter system $(W,S)$ arises as the Weyl group of a symmetrisable Kac--Moody Lie algebra if and only if
\begin{enumerate}
\item $m_{st}\in\{2,3,4,6,\infty\}$ for all $s\neq t$ in $S$, and
\item in each circuit of the Coxeter graph not containing the edge label $\infty$, the number of edges labelled 4 (resp. 6) is even.
\end{enumerate}
\end{thm}

\begin{proof}
Let $C=(c_{ij})$ be a symmetrisable generalised Cartan matrix, say $C=D^{-1}B$ with $B$ symmetric and $D$ diagonal. Let $W$ be the corresponding Weyl group, so with simple reflections $s_i$ and exponents $m_{ij}$ given by the table
\begin{center}
\begin{tabular}{c|ccccc}
\toprule
$c_{ij}c_{ji}$ & $0$ & $1$ & $2$ & $3$ & $\geq4$\\
$m_{ij}$      & $2$ & $3$ & $4$ & $6$ & $\infty$\\
\bottomrule
\end{tabular}
\end{center}
Suppose we have a circuit in the Coxeter graph of $W$ not containing the edge label $\infty$, say with vertices $i_1,\ldots,i_n$ ordered cyclically. Then the product
\[ \prod_l(c_{ll+1}c_{l+1l}) = \prod_l(c_{ll+1}^2d_l/d_{l+1}) = \prod_lc_{ll+1}^2 \]
is a square. It follows that the number of 2s (resp. 3s) is even. Hence the number of edges labelled 4 (resp. 6) is even. Thus condition (2) holds, and we already know that condition (1) holds.

Conversely, let $(W,S)$ be a Coxeter system satisfying conditions (1) and (2). We first define the diagonal matrix $D$. Ignore all edges in the Coxeter graph having label $\infty$. Then for each connected component, chose any vertex $i$ and set $d_i:=1$. If $j$ is another vertex in the same component, then there is a path in the Coxeter graph from $i$ to $j$ not containing $\infty$ as an edge label; set $d_j:=2^a3^b$, where $a$ is the number of 4s in the path modulo 2, and $b$ is the number of 6s in the path modulo 2. By condition (2) this number is independent of the chosen path.

We now define the matrix $C$. Given an edge $\begin{tikzcd}[column sep=20pt]i \ar[-]{r}{m} &j\end{tikzcd}$ in the Coxeter graph, define the pair $(c_{ij},c_{ji}):= (-l/d_i,-l/d_j)$, where
\[ l := \begin{cases} 0 & m=2;\\\mathrm{lcm}(d_i,d_j) & m=3,4,6;\\2\,\mathrm{lcm}(d_i,d_j) & m=\infty. \end{cases} \]
Note that if $m=3,4,6$, then from the construction of the matrix $D$ we must have $\mathrm{lcm}(d_i,d_j)=\max(d_i,d_j)$ and $c_{ij}c_{ji}=m$. Otherwise, if $m=\infty$, then $c_{ij}c_{ji}\geq4$. We conclude that $C$ is a generalised Cartan matrix, that $DC$ is symmetric, and that the corresponding Weyl group is isomorphic to $W$.
\end{proof}

\end{appendix}


\begin{thebibliography}{99}
%
\bibitem{Ar2009} D. Armstrong, Generalized noncrossing partitions and combinatorics of Coxeter groups,
Mem. Amer. Math. Soc. {\bf 202} (2009), no.~949, x+159 pp.
%
\bibitem{BDSW14} B. Baumeister, M. Dyer, C. Stump\ and\ P. Wegener,
A note on the transitive Hurwitz action on decompositions of parabolic Coxeter elements,
Proc. Amer. Math. Soc. Ser. B 1 (2014), 149--154.
%
\bibitem{Be1991} D. J. Benson, {\it Representations and cohomology. I}, second edition,
Cambridge Studies in Advanced Mathematics {\bf 30}, Cambridge Univ. Press, Cambridge, 1998.
%
\bibitem{Be2003} D. Bessis, The dual braid monoid,
Ann. Sci. \'Ecole Norm. Sup. (4) {\bf 36} (2003), no.~5, 647--683.
%
\bibitem{BR2011} D. Bessis\ and\ V. Reiner, Cyclic sieving of noncrossing partitions for complex reflection groups,
Ann. Comb. {\bf 15} (2011), no.~2, 197--222.
%
\bibitem{Bo1990} A. I. Bondal, Representations of associative algebras and coherent sheaves (Russian),
Izv. Akad. Nauk SSSR Ser. Mat. {\bf 53} (1989), no. 1, 25--44;
translation in Math. USSR-Izv. {\bf 34} (1990), no.~1, 23--42.
%
\bibitem{Br2001} T. Brady, A partial order on the symmetric group and new $K(\pi,1)$'s for the braid groups,
Adv. Math. {\bf 161} (2001), no.~1, 20--40.
%
\bibitem{BW2002} T. Brady\ and\ C. Watt, $K(\pi,1)$'s for Artin groups of finite type,
Geom. Dedicata {\bf 94} (2002), 225--250.
%
\bibitem{BMRT07} A. Buan, R. Marsh, I. Reiten and G. Todorov, Clusters and seeds in acyclic cluster algebras,
Proc. Amer. Math. Soc. {\bf 135}  (2007), no.~10, 3049--3060 (with appendix by the above authors, P. Caldero, and B. Keller).
%
\bibitem{BRT2012} A. B. Buan, I. Reiten\ and\ H. Thomas,
From $m$-clusters to $m$-noncrossing partitions via exceptional sequences,
Math. Z. {\bf 271} (2012), no.~3-4, 1117--1139.
%
\bibitem{Ca1972} R. W. Carter, Conjugacy classes in the Weyl group,
Compositio Math. {\bf 25} (1972), 1--59.
%
\bibitem{CB1992} W. Crawley-Boevey, Exceptional sequences of representations of quivers, in
{\it Proceedings of the Sixth International Conference on Representations of Algebras (Ottawa, ON, 1992)},
117--124, Carleton-Ottawa Math. Lecture Note Ser. {\bf 14}, Carleton Univ., Ottawa, ON, 1992.
%
\bibitem{De1974} P. Deligne, Letter to E. Looijenga, 9/3/1974. Available at
  \spverb|http://homepage.univie.ac.at/ christian.stump/Deligne_Looijenga_Letter_09-03-1974.pdf|.
%
\bibitem{Deo1989} V. V. Deodhar, A note on subgroups generated by reflections in Coxeter groups,
Arch. Math. {\bf 53} (1989), 543--546.
%
\bibitem{DW2002} H. Derksen\ and\ J. Weyman, On the canonical decomposition of quiver representations,
Compositio Math. {\bf 133} (2002), 245--265.
%
\bibitem{DR1974} V. Dlab\ and\ C. M. Ringel, Indecomposable representations of graphs and algebras,
Mem. Amer. Math. Soc. {\bf 6} (1976), no. 173, v+57 pp.
%
\bibitem{Dy1990} M. J. Dyer, Reflection subgroups of Coxeter systems,
J. Algebra {\bf 135} (1990), 57--73.
%
\bibitem{Dy2001} M. J. Dyer, On minimal lengths of expressions of Coxeter group elements as products of reflections,
Proc. Amer. Math. Soc. {\bf 129} (2001), no.~9, 2591--2595.
%
\bibitem{FZ02} S. Fomin and A. Zelevinsky, Cluster algebras I: foundations,
J. Amer. Math. Soc. {\bf 15} (2002), 497--529.
%
\bibitem{Ga1973} P. Gabriel, Indecomposable representations. II, in
{\it Symposia Mathematica, Vol. XI (Convegno di Algebra Commutativa, INDAM, Rome, 1971)},
81--104, Academic Press, London, 1973.
%
\bibitem{Ga1981} P. Gabriel, Un jeu? Les nombres de Catalan, in
{\it Z\"urich Uni, Mitteilungsblatt des Rektorats, Dezember 1981}. Available at\\
\spverb|www.math.uni-bielefeld.de/~hkrause/un-jeu-les-nombres-de-catalan.pdf|.
%
\bibitem{GP1987} P. Gabriel and J. A. de la Pe\~na, Quotients of representation-finite algebras,
Comm. Algebra {\bf 15} (1987), 279--307.
%
\bibitem{GL1991} W. Geigle\ and\ H. Lenzing, Perpendicular categories with applications to representations and sheaves,
J. Algebra {\bf 144} (1991), no.~2, 273--343.
%
\bibitem{Go1988} A. L. Gorodentsev, Exceptional bundles on surfaces with a moving anticanonical class (Russian),
Izv. Akad. Nauk SSSR Ser. Mat. {\bf 52} (1988), no. 4, 740--757, 895;
translation in Math. USSR-Izv. {\bf 33} (1989), no.~1, 67--83.
%
\bibitem{GR1987} A. L. Gorodentsev\ and\ A. N. Rudakov, Exceptional vector bundles on projective spaces,
Duke Math. J. {\bf 54} (1987), no.~1, 115--130.
%
\bibitem{Ha1987} D. Happel, On the derived category of a finite-dimensional algebra,
Comment. Math. Helv. {\bf 62} (1987), no.~3, 339--389. 
%
\bibitem{HR1982} D. Happel\ and\ C. M. Ringel, Tilted algebras,
Trans. Amer. Math. Soc. {\bf 274} (1982), no.~2, 399--443.
%
\bibitem{Ho1982} R. B. Howlett, Coxeter groups and $M$-matrices,
Bull. London Math. Soc. {\bf14} (1982), 137--141.
%
\bibitem{Hu2011} A. W. Hubery, The cluster complex of an hereditary artin algebra,
Algebr. Represent. Theor. {\bf 14} (2011), no.~6, 1163--1185.
%
\bibitem{Hu} A. W. Hubery, Ringel-Hall Algebras, lecture notes,
\verb|www1.maths.leeds.ac.uk/~ahubery/|.
%
\bibitem{Hum1990} J. E. Humphreys, {\it Reflection Groups and Coxeter Groups},
Cambridge Studies in Advanced Mathematics {\bf 29}, Cambridge Univ. Press, Cambridge, 1990
%
\bibitem{IT2009} C. Ingalls\ and\ H. Thomas, Noncrossing partitions and representations of quivers,
Compos. Math. {\bf 145} (2009), no.~6, 1533--1562.
%
\bibitem{IS2010} K. Igusa\ and\ R. Schiffler, Exceptional sequences and clusters,
J. Algebra {\bf 323} (2010), no.~8, 2183--2202.
%
\bibitem{Ka1990} V. G. Kac, {\it Infinite dimensional Lie algebras}, third edition,
Cambridge Univ. Press, Cambridge, 1990.
%
\bibitem{Ke1996} O. Kerner, Representations of wild quivers, in
{\it Representation theory of algebras and related topics (Mexico City, 1994)},
65--107, CMS Conf. Proc. {\bf 19}, Amer. Math. Soc., Providence, RI, 1996.
%
\bibitem{Kr1972} G. Kreweras, Sur les partitions non crois\'ees d'un cycle,
Discrete Math. {\bf 1} (1972), no.~4, 333--350.
%
\bibitem{KS2010} H. Krause\ and\ J. \v S\v tov\'\i \v cek, The telescope conjecture for hereditary rings via Ext-orthogonal pairs,
Adv. Math. {\bf 225} (2010), no.~5, 2341--2364.
%
\bibitem{Le1996} H. Lenzing, A $K$-theoretic study of canonical algebras, in
{\it Representation theory of algebras (Cocoyoc, 1994)},
433--454, CMS Conf. Proc. {\bf 18}, Amer. Math. Soc., Providence, RI, 1996.
%
\bibitem{MRZ03} R. Marsh, M. Reineke and A. Zelevinsky, Generalized associahedra via quiver representations,
Trans. Amer. Math. Soc. {\bf 355} (2003), 4171--4186.
%
\bibitem{ONSFR2013} M. Obaid, K. Nauman, W. S. M. Al-Shammakh, W. Fakieh and C. M.  Ringel,
The number of complete exceptional sequences for a Dynkin algebra,
Colloq. Math. {\bf 133} (2013), no.~2, 197--210.
%
\bibitem{ORT2012} S. Oppermann, I. Reiten and H. Thomas, Quotient closed subcategories of quiver representations,
arXiv:1205.3268.
%
\bibitem{Re2007} N. Reading, Clusters, Coxeter-sortable elements and noncrossing partitions,
Trans. Amer. Math. Soc. {\bf 359} (2007), no.~12, 5931--5958.
%
\bibitem{Rie1980} C. Riedtmann, Representation-finite self-injective algebras of class $A_n$, in
{\it Representation theory, II (Proc. Second Internat. Conf., Carleton Univ., Ottawa, Ont., 1979)},
449--520, Lecture Notes in Math. {\bf 832}, Springer, Berlin, 1980.
%
\bibitem{RS1990} C. Riedtmann\ and\ A. Schofield, On open orbits and their complements,
J. Algebra {\bf 130} (1990), no.~2, 388--411.
%
\bibitem{Ri1976} C. M. Ringel, Representations of $K$-species and bimodules,
J. Algebra {\bf 41} (1976), no. 2, 269--302. 
%
\bibitem{Ri1983} C. M. Ringel, Bricks in hereditary length categories,
Resultate Math {\bf 6} (1983), 64--70.
%
\bibitem{Ri1994} C. M. Ringel, The braid group action on the set of exceptional sequences of a hereditary Artin algebra, in
{\it Abelian group theory and related topics (Oberwolfach, 1993)},
339--352, Contemp. Math. {\bf 171}, Amer. Math. Soc., Providence, RI, 1994.
%
\bibitem{Ri1996} C. M. Ringel, Exceptional objects in hereditary categories, in
{\it Representation theory of groups, algebras, and orders (Constanta, 1995)},
An. St. Univ. Ovidius Constanta, Ser. Mat. {\bf 4} (1996), no. 2, 150--158.
%
\bibitem{Ru1990} A. N. Rudakov et al., {\it Helices and vector bundles: Seminaire Rudakov},
London Mathematical Society Lecture Note Series {\bf 148}, Cambridge Univ. Press, Cambridge, 1990.
%
\bibitem{Sc1990} A. Schofield, The internal structure of real Schur representations, preprint (1990).
%
\bibitem{Sc1991} A. Schofield, Semi-invariants of quivers,
J. London Math. Soc. (2) {\bf 43} (1991), no.~3, 385--395.
%
\bibitem{Sc1992} A. Schofield, General representations of quivers,
Proc. London Math. Soc. (3) {\bf 65} (1992), no.~1, 46--64.
%
\bibitem{Si2012} R. C. Sim\~oes, Hom-configurations and noncrossing partitions,
J. Algebraic Combin. {\bf 35} (2012), no.~2, 313--343.
%
\end{thebibliography}
\end{document}